\documentclass[a4paper]{amsart} 
\pdfoutput=1

\usepackage[english]{babel}
\usepackage{multicol}
\usepackage{amsmath}
\usepackage{amssymb}
\usepackage{amsfonts}
\usepackage{mathrsfs}
\usepackage{graphicx}
\usepackage{amsthm}
\usepackage{fancyhdr}
\usepackage{newlfont}
\usepackage{indentfirst}
\usepackage{geometry}
\usepackage{enumerate}
\usepackage{color}
\usepackage[font={small,it}]{caption}
\usepackage[foot]{amsaddr}

\newtheorem{defi}{Definition}
\newtheorem{theo}{Theorem}

\newtheorem{propo}{Proposition}

\theoremstyle{remark}
\newtheorem*{rem}{Remark}

\newtheorem*{ex}{Example}

\numberwithin{figure}{section}
\numberwithin{equation}{section}

\begin{document}
 
 \title{A metric model for the functional architecture of the visual cortex}
 \date{}
\author{Noemi Montobbio$^\ast$}
\address{$^\ast$Dipartimento di Matematica, Universit\`{a} di Bologna, Italy. noemi.montobbio2@unibo.it}
\author{Alessandro Sarti$^\dagger$}
\address{$^\dagger$CAMS, CNRS - EHESS, Paris, France. alessandro.sarti@ehess.fr}
\author{Giovanna Citti$^\ddagger$}
\address{$^\ddagger$Dipartimento di Matematica, Universit\`{a} di Bologna, Italy. giovanna.citti@unibo.it}

\begin{abstract}
 The purpose of this work is to construct a model for the functional architecture of the primary visual cortex (V1), based on a structure of \emph{metric measure space} induced by the underlying organization of receptive profiles (RPs) of visual cells. In order to account for the horizontal connectivity of V1 in such a context, a diffusion process compatible with the geometry of the space is defined following the classical approach of K.-T. Sturm \cite{sturm98}. The construction of our distance function does neither require any group parameterization of the family of RPs, nor involve any differential structure. As such, it adapts to non-parameterized sets of RPs, possibly obtained through numerical procedures; it also allows to model the lateral connectivity arising from non-differential metrics such as the one induced on a pinwheel surface by a family of filters of vanishing scale. On the other hand, when applied to the classical framework of Gabor filters, this construction yields a distance approximating the sub-Riemannian structure proposed as a model for V1 by Citti and Sarti \cite{cs06}, thus showing itself to be consistent with existing cortex models.
\end{abstract}

\maketitle

\section*{Introduction}\label{intro}

 The primary visual cortex (V1) is the first cortical area which receives the visual signal from the retina. The first celebrated description of its structure dates back to the '60s, when Hubel and Wiesel discovered \cite{HW} (see also \cite{hubel}) that cortical neurons are not only sensible to the intensity of the visual stimulus, but also to other variables,  called ``engrafted'', such as orientation, scale, velocity. Precisely, every retinal location is associated to a whole set (called hypercolumn) of cells of V1, sensitive to all the possible values of the considered feature. The first processing of a visual stimulus in V1 is performed by a class of neurons called \emph{simple cells}. The activation of a simple cell in response to an image $I(x,y)$ on the retinal plane can be modeled as a linear integral operator with associated kernel $\psi(x,y)$, called the \emph{receptive profile} (RP) of the neuron. This means that the RPs of simple cells can be represented, up to a first approximation, by means of a set $\{\psi_p\}_{p \in \mathcal{G}} \subseteq L^2(\mathbb{R}^2)$ of \emph{linear filters} on the plane. Such a family \emph{lifts} the image to the set $\mathcal{G}$ of parameters encoding the features extracted by the filters. Typically, this set is of the form $\mathcal{G} = \mathbb{R}^2 \times \mathcal{F}$, where $(x,y) \in \mathbb{R}^2$ denotes the point of the retina on which the profile is centered (typically with a strongly concentrated support), thus encoding the feature of \emph{position}, while $\Phi \in \mathcal{F}$ expresses the engrafted variables.\\
A well-estabilished model for the RPs of V1 simple cells is represented by Gabor filters (\cite{jonpal}, \cite{daug}, \cite{lee}): the whole bank of filters $\{\psi_{x,y,\theta}\}_{x,y,\theta}$ is obtained by translations of $(x,y) \in \mathbb{R}^2$ and rotations of $\theta \in S^1$ of a \emph{mother function}
\[
 \psi(u,v) = \exp\left(\frac{2\pi i u}{\lambda}\right) \exp\bigg(-\frac{u^2+v^2}{2\sigma^2}\bigg).
\]
Thus, the corresponding \emph{feature space} is $\mathcal{G}=\mathbb{R}^2\times S^1$, encoding \emph{position} and \emph{orientation}.

 The neural activity is known to propagate across V1 through the so-called \emph{horizontal connections}, linking neurons sensitive to similar orientations but belonging to different hypercolumns \cite{gilwie}. The spatial extent and the marked orientation specificity of such connections have been investigated in a number of experiments (see e.g. \cite{gilwie89}, \cite{bosking}). These properties are believed (\cite{gilbert}, \cite{petitond}) to be the neurophysiological counterpart to the perceptual rules expressed by the concept of \emph{association field}, introduced by Field, Hayes and Hess in 1993 \cite{field} to describe the results of their psychophysical experiments on contour integration.\\
 Through the past twenty years, a number of models were proposed that describe the functional architecture of V1 through differential structures. See e.g. \cite{hoffman}, \cite{petitond}, \cite{zucker}, \cite{cs06}, \cite{scp}. See also \cite{neuromat} for a review. In these models, V1 is represented as a feature space, typically endowed with a Lie group structure. For instance, the \emph{rototranslation group} $\mathbb{R}^2\times S^1$ is taken into consideration in \cite{cs06}. In this work, $\mathbb{R}^2\times S^1$ is endowed with a sub-Riemannian structure which is invariant with respect to the group law. The spreading of neural activity in V1 through the lateral connectivity is described by means of a propagation along the integral curves of this structure.
 
 The aim of this work is to propose a model of V1 as a metric measure space whose structure is induced directly by the RPs of simple cells. This suggests that the geometrical rules controlling the intracortical connections of V1 may be recovered from the shape of such RPs.\\
 Our definition of the \emph{cortical metric space} is straightforward. V1 is represented by the family $\mathcal{G}$ of parameters indexing a bank of filters $\{\psi_p\}_{p \in \mathcal{G}}$. The distance between two points $p_0,p_1 \in \mathcal{G}$ is defined as
 \[
  d(p_1,p_0) := \|\psi_{p_1}-\psi_{p_0}\|_{L^2(\mathbb{R}^2)}.
 \]
Therefore, the filters do not only provide a set of parameters on which to define a geometric structure, but rather they contribute to the characterization of such a structure. This metric space is then equipped with its associated spherical Hausdorff measure. Such a construction does not require any invariance or group structure onto the set $\mathcal{G}$ indexing the RPs: the distance $d$ would still be well defined even for a non-parameterized set of filters known numerically. \\
 As for the propagation along the horizontal connectivity, the idea is still to consider a diffusion process, associated to a suitable operator which must play, in this setting, a role analogous to that of the Laplace-Beltrami operator in the differential case. To this end, we will refer to the classical approach of K.-T. Sturm (\cite{sturm95b}, \cite{sturm98}), which provides a general method to construct a diffusion process on a metric measure space $(X,d,\mu)$. This technique consists of defining a Dirichlet form on $L^2(X,\mu)$ whose associated positive self-adjoint operator has a heat kernel admitting Gaussian estimates in terms of the distance $d$, provided that a Measure Contraction Property (MCP, see Definition 3) on the space is satisfied.
 
 We will give all the details with regard to the feature space determined by a family of Gabor filters. This example is very meaningful for two reasons. First, it is useful in terms of intuition and manageability, since the invariances of the feature space in this setting make it possible to perform some explicit calculations (it is nevertheless important to notice that these invariances are not taken into account in the construction of the metric space in the general case). Second, it links the present metric model to the existing differential models: indeed, the distance function obtained in this case turns out to be locally equivalent to a Riemannian distance which approximates the sub-Riemannian structure defined on $\mathbb{R}^2\times S^1$ in the model presented in \cite{cs06}.\\
 As a motivation behind the choice of the general setting of metric measure spaces, we then present a further example. This consists of a sub-Riemannian surface in $\mathbb{R}^2 \times S^1$, obtained as the feature space defined by a sub-family of Gabor filters with vanishing scale: the restriction of the distance onto the surface is indeed not associated to any differential structure. We prove that such a metric measure space satisfies the MCP, thus providing a model for lateral connectivity in environments such as pinwheel surfaces.
 
 In the last part of the paper, we propose to approximate the propagation along the horizontal connectivity through repeated integrations against a kernel, computed explicitly in terms of the distance function, which estimates the heat kernel of the diffusion process for small times. With a view to dealing with non-parameterized banks of filters, possibly obtained through numerical procedures, it is indeed desirable to dispose of an explicit algorithm to compute the cortical connectivity associated to them. We present some numerical simulations, comparing the propagations obtained respectively through a discretized heat equation and through repeated integrations against such kernel, in the case of Gabor filters. We refer to our parallel paper \cite{neuro} for a more extensive discussion on this approach.

\section{Background}\label{background}

\subsection{Receptive profiles and simple cells}\label{simple}

 We first provide the necessary background on some structures at the basis of the visual system. The visual pathways start from the \emph{retina}, from which the visual signal is conveyed through the optic nerve to the \emph{lateral geniculate nucleus} (LGN). This structure is the main central conjunction to the occipital lobe, in particular to the \emph{primary visual cortex} (V1). From V1, different specialized parallel pathways depart, leading to higher cortical areas performing further processing.

Through the above-mentioned connections, each cell is linked to a specific domain $D$ of the retina which is referred to as its \emph{receptive field} (RF). A retinal cell in the RF can react in an excitatory or in an inhibitory way to a punctual luminous stimulation, with different modulation, and the function $\psi : D \rightarrow \mathbb{R}$ which measures the reaction of the neuron at every retinal location $(x,y)$ is called \emph{receptive profile} (RP).
 Certain types of visual neurons are shown to act, at least to a first approximation, as \emph{linear filters} on the optic signal. This means that the response of the cell to a visual stimulus $I$, defined as a function on the retina, is given by
\begin{equation}\label{linfilt}I_\psi := \int_D I(x,y)\psi(x,y)dxdy.\end{equation}

The shape of the RP of a neuron contains information about the \emph{features} that it extracts from a visual signal. For example, the local support of $\psi$ makes it sensitive to \emph{position}, i.e. the neuron only responds to stimuli in a localized region of the image. Or again, a receptive profile with an elongated shape will be sensitive to a certain \emph{orientation}, i.e. it will respond strongly to stimuli consisting of bars collinear with this shape. If we denote the whole set of RPs by $\{\psi_p\}_{p \in \mathcal{G}}$, where $\mathcal{G}$ is a set of indices, we may regard each $p \in \mathcal{G}$ as representing the features extracted by the filter $\psi_p$: in these terms, we shall refer to $\mathcal{G}$ as the \emph{feature space} associated to the bank of filters $\{\psi_p\}_p$.\\
It is a classical result of neurophysiology that the RPs of LGN cells are best modeled as Laplacians of Gaussians.
As for the primary visual cortex, two main classes of cells can be observed in this area. These neurons are referred to as \emph{simple} and \emph{complex} cells, and they were first discovered by Hubel and Wiesel in the '60s \cite{HW}.\\
\begin{figure}[!htbp]
\centering
 \includegraphics[width=.8\textwidth]{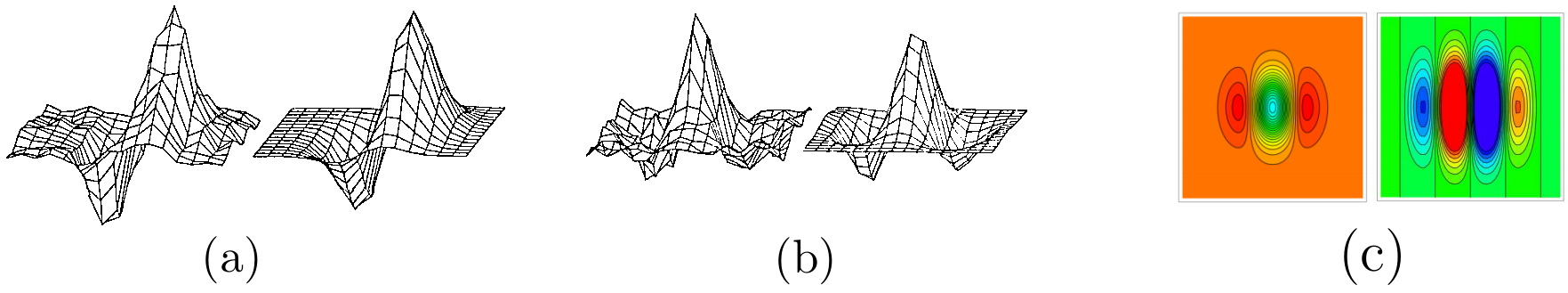}
 \caption{(a) Left: an example of experimentally measured odd RPs of simple cells in cat V1. Right: the best-fitting 2D Gabor function for the cell's RP. (b) The same as (a) for an even RP. Both examples are taken from \cite{daug}. (c) A quadrature pair of Gabor RPs, given by the real (left) and imaginary (right) parts of a complex Gabor function. Source: \cite{petitot}.} \label{RP_V1}
\end{figure}
 Simple cells are the first neurons in the visual pathway showing orientation selectivity, given by a strongly anisotropic RP, as shown in Figure \ref{RP_V1}. They receive most of the outgoing projections from the LGN: it is presumed that each simple receptive field arises from multiple isotropic LGN receptive fields converging in a line \cite{hubel}.
 The set of RPs of V1 simple cells has classically been modeled \cite{jonpal,daug,lee} (see Figure \ref{RP_V1}a-b) through a bank of \emph{Gabor filters} $\{\psi_{x,y,\theta}\}_{x,y,\theta}$, obtained by translations $T_{(x,y)}$ of $(x,y) \in \mathbb{R}^2$ and rotations $R_\theta$ of $\theta \in S^1$ of a \emph{mother filter} $\psi_{0,0,0}$: 
\begin{equation}\label{mother} \psi_{x,y,\theta}(u,v)= \psi_{0,0,0}\left(T^{-1}_{(x,y)}R^{-1}_{\theta}(u,v)\right), \quad \psi_{0,0,0} (u,v) = \exp\left(\frac{2\pi i u}{\lambda}\right) \exp\bigg(-\frac{u^2+v^2}{2\sigma^2}\bigg).
 \end{equation}
 Note that these are complex-valued functions: each filter $\psi_{x,y,\theta}$ actually represents two RPs, given by its real and imaginary parts, sharing the same orientation but shifted by $90^\circ$ in phase. These are referred to as a \emph{quadrature pair} of cells. Real and imaginary parts of Gabor filters represent so-called \emph{even} and \emph{odd} cells respectively (see Figure \ref{RP_V1}c).\\
 In this case, the feature space is $\mathbb{R}^2 \times S^1$: $(x,y) \in \mathbb{R}^2$ encodes the position at which the filter is centered and $\theta \in S^1$ expresses its preferred orientation. For the sake of simplicity and legibility, we take the scale $\sigma$ to be fixed, but it may be let vary as well, yielding a richer feature space.\\
 The information extracted by simple cells is believed to determine the behavior of complex cells, which perform a second order analysis: in particular, according to the \emph{energy model} \cite{complex}, the response of each complex cell is modeled as the square sum of a quadrature pair of simple cells.
 This leads to the phase invariance of these neurons, whose behavior cannot be described through linear filtering.

 \subsection{Horizontal connections and association fields}\label{horizontal}
 
 It has been shown \cite{HW}, through recording of the responses to certain stimuli (e.g. oriented bars passing through the RF), that the preferred retinal position and orientation of V1 neurons are roughly constant moving perpendicularly to the cortical surface. On the other hand, the preferred orientation varies gradually in the directions parallel to the surface, giving rise to groupings called \emph{orientation hypercolumns}: these contain cells sensitive to approximately the same retinal position, but span all orientations, which are ``engrafted'' onto the positional map with a finer subdivision \cite{hubel}. The intracortical circuitry can be described in terms of two main mechanisms: a short-range connectivity taking place within each hypercolumn, which essentially selects the orientation of maximum output in response to a visual stimulus and suppresses the others; and a long-range ``horizontal'' (or lateral) connectivity, connecting neurons belonging to different hypercolumns but sensitive to similar orientations. The latter is also shown \cite{bosking} to spread around each neuron along the axis of its preferred orientation, and it is believed to be at the basis of the ability of our visual system to perform perceptual grouping (\cite{gilbert}, \cite{petitond}). Indeed, a \emph{global} analysis is necessary in order to correctly recognize objects and interpret a visual scene: single receptive profiles alone cannot account for such non-local features.
 
 \begin{figure}[h]
 \centering
 \includegraphics[width=.8\textwidth]{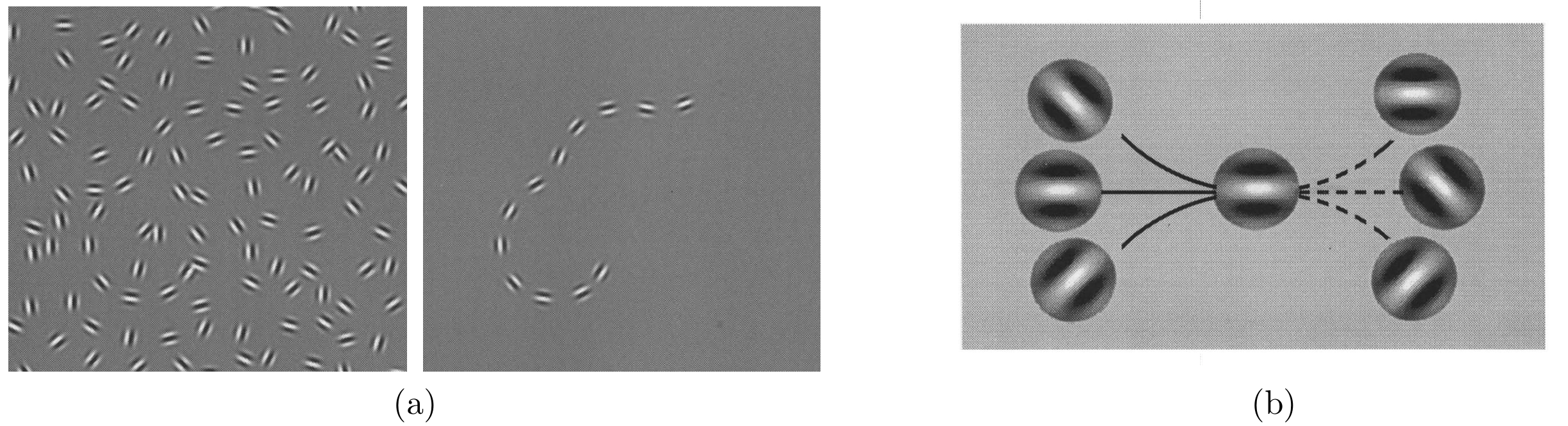}
 \caption{(a) Path segregation. (b) A schematic representation of the notion of association field. Images taken from \cite{field}. \label{perceptual}}
 \end{figure}
 The processing mechanism taking place throughout the visual pathways allows to efficiently group local items into extended contours, and to segregate a path of elements from its background (see Figure \ref{perceptual}a).
 These perceptual mechanisms in V1 have been described through the concept of \emph{association field} \cite{field}, a schematic representation of which is displayed in Figure \ref{perceptual}b: this abstract object characterizes the geometry of the mutual influences between V1 cells depending on their orientation and reciprocal position. In other words, the excitation of a neuron is strengthened by the activation of surrounding cells with certain \emph{relative features} with respect to it. In particular, the strongest correlation takes place between those edge elements that are either \emph{collinear} or \emph{co-circular}. The psychophysical analysis performed in \cite{field} revealed that such influences link neurons even with markedly separated RFs. The comparable spatial extent of association fields and horizontal connections, together with their shared orientation specificity, make the lateral connectivity a potential anatomical implementation of this perceptual phenomenon.

\subsection{Sub-Riemannian models of V1}\label{subriemannian}

 From a mathematical point of view, the hypercolumnar organization of V1 can be described by saying that at each retinal position there exists a full \emph{fiber} of possible orientations. This idea led to the representation of V1 as a \emph{fiber bundle} whose basis is the space of retinal locations, first introduced by Petitot and Tondut in 1999 \cite{petitond}. Their model yields a 3D Heisenberg group structure. 

 A more complete description, allowing non-equioriented boundaries, was given in \cite{cs06} in terms of a sub-Riemannian structure on the rototranslation group, associated to the bank of filters (\ref{mother}). In this case, for every fixed retinal position $(x,y)$, the maximum of the function $I_\psi(x,y,\theta)$ in the variable $\theta$ is attained at a point $\Theta(x,y)$ which represents the orientation of the level lines of the image $I$ at the point $(x,y)$. The images of these lines through the map $(x,y) \mapsto (x,y,\Theta(x,y))$ are called \emph{lifted level lines}, and their tangent vector at every point can be written as a linear combination of the vector fields
  \begin{equation}
  \label{vecY} Y_1 = -\sin\theta \partial_x + \cos\theta \partial_y \: , \quad Y_2 = \partial_\theta.
  \end{equation}
 These vector fields define a bidimensional subbundle of the tangent bundle to $\mathbb{R}^2 \times S^1$, referred to as the \emph{horizontal tangent bundle}, thus determining a sub-Riemannian structure on $\mathbb{R}^2 \times S^1$. The Lie algebra generated by $Y_1$ and $Y_2$ through the bracket operation between vector fields is the whole Euclidean tangent plane, since $[Y_1,Y_2] = \cos(\theta)\:\partial_x + \sin(\theta)\:\partial_y =: Y_3$.
 In other words, $Y_1$ and $Y_2$ satisfy the H\"{o}rmander rank condition. This leads, by the Chow theorem, to the so-called \emph{connectivity property}: any couple of points in $\mathbb{R}^2 \times S^1$ can be connected through a \emph{horizontal curve}, i.e. an integral curve of a section of the horizontal tangent bundle.\\
 The lateral propagation of neural activity in the cortical space is described in \cite{cs06} through the sub-Riemannian heat equation $\partial_t u = \Delta u$, where $\Delta = Y_1^2 + Y_2^2$. In this setting, the association field around a point $(x_0,y_0,\theta_0) \in \mathbb{R}^2 \times S^1$ is then characterized as a family of integral curves of $Y_1$ and $Y_2$ starting at this point. Namely, $\gamma' = Y_{1 \: | \gamma} + k Y_{2 \: | \gamma}$ and $\gamma(0) = (x_0,y_0,\theta_0)$, where $k$ varies in $\mathbb{R}$.\\
 The evolution of the activity of V1 neurons is influenced by a combination of intra-columnar and lateral connections. In \cite{cs06}, the sub-Riemannian diffusion modeling the horizontal connections and the mechanism of selection of maxima implemented by the short-range connectivity have been combined by alternating their action iteratively: precisely, each iteration consists of a first step of diffusion in a finite time interval and a second step of non maximal suppression. The time interval is then sent to zero. See also \cite{bresscow03} and \cite{bresscow02}, where the connections between each couple of neurons are represented by a weight function which is decomposed as the sum of two terms modeling these two mechanisms.\\
 Different diffusion equations in this sub-Riemannian setting, such as the Fokker-Planck equation, have also been used in other works (see e.g. \cite{mumford}, \cite{edge-stat} for a stochastic point of view).

\section{The space of features as a metric space}\label{model}

 In this section we outline our model, whose basic idea is the construction of a metric space encoding the local geometry of the cortex, defined by a notion of ``local correlation'' between RPs of simple cells. As in the differential models described above, we will then characterize the lateral connectivity through a propagation with respect to the metric structure. The space on which the distance function will be defined is the \emph{feature space} $\mathcal{G}$ indexing a family of filters $\{\psi_p\}_{p\:\in\:\mathcal{G}}$ chosen to model the RPs of V1 simple cells. As remarked above, in the case of a family of Gabor filters of fixed scale the feature space is $\mathbb{R}^2 \times S^1$. In effect, we will show that the distance function induced by Gabor filters on this space is locally equivalent to a Riemannian distance on $\mathbb{R}^2 \times S^1$ approximating the Carnot-Carath\'{e}odory distance associated to the sub-Riemannian structure defined in \cite{cs06}.\\
 The main new feature of our model is that the cortical geometry is directly induced by the RPs. This means that any bank of filters used to represent a family of RPs can be given an associated connectivity pattern through this technique.

\subsection{The cortical distance}\label{distance}

As said earlier, the family of RPs of simple cells of V1 can be modeled by a bank of linear filters $\{\psi_p\}_{p\:\in\:\mathcal{G}} \subseteq L^2(\mathbb{R}^2)$. In this section, we define a metric structure on the set of parameters $\mathcal{G}$ associated to such a family of filters.
\begin{defi}
 Let $\{\psi_p\}_{p\:\in\:\mathcal{G}}$ be a family of real- or complex-valued functions in $L^2(\mathbb{R}^2)$. We call $\mathcal{G}$ the \emph{feature space} associated to the family $\{\psi_p\}$.\\
 We then define the distance function $d \: : \: \mathcal{G}\times \mathcal{G} \longrightarrow \mathbb{R}$,
 \begin{equation}\label{kerdist} d(p,p_0) := \|\psi_p-\psi_{p_0}\|_{L^2(\mathbb{R}^2)},\end{equation}
 and the \emph{generating kernel} $K \: : \: \mathcal{G}\times \mathcal{G} \longrightarrow \mathbb{R}$,
 \begin{equation}
  K(p,p_0) := Re\langle\psi_p,\psi_{p_0}\rangle_{L^2}.
 \end{equation}
\end{defi}
 The introduction of the kernel $K$ was inspired by the definition of the reproducing kernel induced by a family of admissible wavelets on the image of the associated wavelet transform (see \cite{rkwavelet}, \cite{deng}). \\
 Note that $d^2 (p, p_0)  = \| \psi_p - \psi_{p_0} \|_{L^2}^2 = \|\psi_{p}\|_{L^2}^2 + \|\psi_{p_0}\|_{L^2}^2 - 2Re\langle\psi_p,\psi_{p_0}\rangle_{L^2}$.
Since we can assume the filters to be normalized to have $L^2$-norm equal to $t$, the above expression only depends on the real part of the inner product between the two filters, that is on the kernel $K$:
\begin{equation}\label{dK}
   d^2 (p, p_0)  = 2t - 2K(p,p_0).
\end{equation}
 $K$ can be thought of as a connectivity kernel, expressing the strength of \emph{correlation} between two profiles. Of course, the distance between two points increases as this correlation fades.

\subsection{Local distance and gluing}\label{gluing}

 The function we defined is obviously a distance on $\mathcal{G}$, since it is a restriction of the $L^2$ distance function. However, one may want to introduce some constraints on which filters can directly interact with one another in determining the geometry of the space -- for instance, this can be done to inspect the behavior of the connectivity w.r.t. certain features encoded in the RPs. We will see a concrete example of this situation in the case of Gabor filters, where we will be able to \emph{isolate} the spreading of neural activity along the axis of the preferred orientation of the starting RP, while discarding the contributions along the orthogonal axis.\\
 Imposing such constraints corresponds to defining around each point $p_0 \in \mathcal{G}$ a \emph{local patch} $\mathcal{P}(p_0)\subseteq\mathcal{G}$, and to restrict the definition of $d$ to this set. The following question arises naturally: is it possible to glue all these local distances together to obtain a global distance function on the feature space? In order to get to this result, we will need to make one further assumption on the patches $\mathcal{P}(p_0)$. We now define a new function $\tilde{d}$ as follows.
 \begin{defi}\label{defdist}
 For $p,p_0 \in \mathcal{G}$, we set
 \begin{align}\label{globdist}
 \tilde{d}(p,p_0) := \inf \left\{ \:\sum_{j=1}^{N} d(q_{j-1},q_j) \: : \: N \in \mathbb{N},\: q_0=p_0,\:q_N=p,\:q_j \in \mathcal{P}(q_{j-1}) \:\forall j \: \right\}.
 \end{align}
 \end{defi}
 Note that, in general, the existence of a sequence $\{q_j\}_{j=1,...,N}$ such that $q_0=p_0,\:q_N=p$ and $\:q_j \in \mathcal{P}(q_{j-1}) \:\forall j=1,...,N$ is not guaranteed for any couple of points $(p_0,p)$. If such a sequence does not exist, we consider the distance between the two points to be infinite. However, this would be a ``degenerate'' case where there are isolated points or regions of the feature space, corresponding to neurons whose activations are mutually independent.
\begin{propo}\label{dtilde}
 Given a set $\mathcal{G}$, define around each point $p_0$ a \emph{patch} $\mathcal{P}(p_0) \subseteq \mathcal{G}$ such that
 \begin{equation}\label{patchball}
  \forall p_0 \in \mathcal{G} \quad \exists \varepsilon>0 \; : \; B_\varepsilon(p_0) := \{p \in \mathcal{G} \: : \: d(p,p_0)<\varepsilon \} \subseteq \mathcal{P}(p_0). \end{equation}
 Then $\tilde{d} \: : \: \mathcal{G}\times \mathcal{G} \longrightarrow \mathbb{R}$ defined as above satisfies:
 \begin{enumerate}[(i)]
  \item $\tilde{d}(p,q) \geq 0 \quad \forall p,q \in \mathcal{G}$,
  \item $\tilde{d}(p,s)+\tilde{d}(s,q) \geq \tilde{d}(p,q) \quad \forall p,s,q \in \mathcal{G}$,
  \item $\forall p,q \in \mathcal{G}, \; \tilde{d}(p,q)=0 \: \Leftrightarrow \: p=q$.
 \end{enumerate}
\end{propo}
\begin{proof}
First, $\tilde{d}$ is well-defined. This means verifying that the local distance functions coincide on overlapping patches. Indeed, this happens by construction, since $d(p,p_0)$ is always equal to the $L^2$ distance between $\psi_p$ and $\psi_{p_0}$.\\
Second, $\tilde{d}$ verifies the properties.
 \begin{enumerate}[(i)]
 \item $\tilde{d}$ is obviously non negative.
 \item As for the triangle inequality, we have: 
 \begin{align*} & \tilde{d}(p,s)+\tilde{d}(s,q) \\ & = \inf \left\{ \:\sum_{j=1}^{N} d(q_{j-1},q_j) \: | \: N \in \mathbb{N},\: q_0=q,\: q_N=p,\:q_j \in \mathcal{P}(q_{j-1}) \:\forall j, \exists j : q_j=s \: \right\} \\
 & \geq \inf \left\{ \:\sum_{j=1}^{N} d(q_{j-1},q_j) \: | \: N \in \mathbb{N},\: q_0=q,\: q_N=p,\:q_j \in \mathcal{P}(q_{j-1}) \:\forall j \: \right\} = \tilde{d}(p,q).
 \end{align*}
 \item Lastly, we have to prove that $\tilde{d}(p,p_0)=0 \: \Leftrightarrow \: p=p_0$.
 Suppose $p\neq p_0$. From (\ref{patchball}), there exists an $\varepsilon>0$ such that $B_\varepsilon(p_0)\subseteq\mathcal{P}(p_0)$. Now,
 \begin{itemize}
 \item if $p \notin \mathcal{P}(p_0)$, then $p \notin B_\varepsilon(p_0)$ and consequently $\tilde{d}(p,p_0)\neq 0$;
 \item on the other hand, if $p$ is in $\mathcal{P}(p_0)$, then $\tilde{d}(p,p_0)\neq 0$ for the properties of $d$, which is a distance on $\mathcal{P}(p_0)$.
 \end{itemize}
 \end{enumerate}
\end{proof}

 \begin{rem}\label{remsymm}
 Given a sequence $p=q_0, q_1,\ldots,q_N=q$, the condition $q_j \in \mathcal{P}(q_{j-1})$ does not imply having $q_{j-1} \in \mathcal{P}(q_j)$. Therefore, in general, Proposition \ref{dtilde} yields an \emph{asymmetric} distance $\tilde{d}$. This intuitively means that, given two points $p$ and $q$, getting from $p$ to $q$ may be harder than getting from $q$ to $p$, i.e. $\tilde{d}(p,q)>\tilde{d}(q,p)$. In practical applications, this could represent for example the situation where $p$ and $q$ are points in space and $q$ is uphill w.r.t. $p$. See \cite{quasimetric} as a reference on \emph{quasimetric spaces}.\\
 However, recall that the distance we are defining should model the lateral connectivity in V1. Due to the evidence that horizontal connections are largely reciprocal \cite{reciprocal}, it is reasonable to model this phenomenon through a symmetric distance. Since the construction of the patches $\mathcal{P}(\cdot)$ was meant to restrict which cells can interact with one another, it is natural to define them so that $p$ is connected to $q$ if and only if $q$ is connected to $p$. This means requiring that $q \in \mathcal{P}(p) \Leftrightarrow p \in \mathcal{P}(q)$, which implies the symmetry of $\tilde{d}$ by considering for each sequence $p=q_0, q_1,\ldots,q_N=q$ the reversed sequence $\{q_{N-j}\}_{j=1,...,N}$. In the following, the symmetry of the distance is taken as an assumption.
 \end{rem}
 
 To sum up, the kernel distance defined in (\ref{kerdist}) may be treated as a local object by restricting it to suitable patches defined around each point. In order to have a meaningful distance on the whole feature space taking into account these constraints, the local distance functions must be glued together: the above Proposition states that, under reasonable conditions on the choice of the patches, this yields a well-defined global distance on $\mathcal{G}$.

\subsection{The case of Gabor filters}\label{gaborcase}

 In this section, we show the results of applying the model described to the classical example of a bank of Gabor filters. We then prove that the distance obtained in this case is locally equivalent to a Riemannian approximation to the sub-Riemannian metric introduced in \cite{cs06}.\\
 Let us consider the set $\{\psi_{x,y,\theta}\}_{x,y,\theta}$ of Gabor filters introduced in (\ref{mother}). For each value of $\lambda>0$ and $\sigma>0$, one obtains a family of filters parameterized by $p=(x,y,\theta) \in \mathbb{R}^2\times S^1$ where each of the filters $\psi_{x,y,\theta}$ has wavelength $\lambda$ and scale $\sigma$.
 
\paragraph{\textbf{The distance function}}\label{gabordist}

 Fix $\lambda,\sigma >0$ and denote $p = (x,y,\theta)$ and $p_0 = (x_0,y_0,\theta_0)$. As anticipated, there is some invariance in the behavior of the kernel $K(p,p_0)=Re\langle \psi_p , \psi_{p_0} \rangle_{L^2(\mathbb{R}^2)}$ in the Gabor case. Specifically, through a straightforward calculation one obtains:
  \[K\big((x,y,\theta), (x_0,y_0,\theta_0)\big) = K\big( (R_{\theta_0}T_{(x_0,y_0)}(x,y),\theta-\theta_0),\; (0,0,0) \big). \]
 It is therefore sufficient to compute explicitly the expression of $K(p,p_0)$ for $p_0 = (0,0,0)$. We have:
 \begin{align*}\langle \psi_p , \psi_{0} \rangle_{L^2(\mathbb{R}^2)} =
  \sigma^2\pi \: \exp\left( -\frac{x^2}{4\sigma^2} -\frac{y^2}{4\sigma^2} - \frac{2\sigma^2\pi^2 (1 - \cos\theta)}{\lambda^2} \right) \: \exp\left( -i \pi \: \frac{x(1 + \cos\theta) + y\sin\theta}{\lambda} \right).
 \end{align*}
 The real part of this scalar product gives the kernel $K$. Of course, the same invariance holds for the distance $d$. Since the squared $L^2$-norm of each of the filters (\ref{mother}) is equal to $\sigma^2\pi$, we have:
 \begin{equation}\label{distgabor}
 d^2 (p,0) = 2\sigma^2\pi - 2\sigma^2\pi \: \exp \left( -\frac{x^2}{4\sigma^2} -\frac{y^2}{4\sigma^2} - \frac{2\sigma^2\pi^2 (1 - \cos\theta)}{\lambda^2} \right) \cdotp \cos\left( \pi \: \frac{x(1 + \cos\theta) + y\sin\theta}{\lambda} \right).
 \end{equation}
 Note that the distance $d$ depends on $\sigma$ and $\lambda$, since the scale and wavelength of the filters naturally influence its spatial extent and oscillatory behavior respectively.

\paragraph{\textbf{Local patches}}\label{gaborpatch}
 
 The balls $B_\varepsilon(p_0) = B_\varepsilon\big((x_0,y_0,\theta_0)\big)$ of the distance $d$ over a certain radius $\varepsilon$ are not connected. This is a consequence of the oscillatory behavior of the distance along the axis orthogonal to the preferred orientation $\theta_0$ of the starting filter $\psi_{p_0}$: the central connected component of the ball contains the points of $\mathcal{G}$ corresponding to filters either collinear or co-circular with $\psi_{p_0}$, while the smaller lateral lobes account for the effect of \emph{parallel} filters. As anticipated, we shall (at least as a first stage) examine only the contribution of the ``principal'' connected component, corresponding to the classical definition of association fields \cite{field}. This is practical in order to compare our model with previous works, although the oscillatory component of the kernel could account for other biologically realistic aspects such as the so-called \emph{ladder effect} \cite{yenfinkel,neuro}. In order to keep only the central lobe, we define the distance function on local patches such that it is truncated where it reaches its maximum. This means ``eliminating'' the periodicity of the cosine in Eq. (\ref{distgabor}) by defining around $(0,0,0)$ a \emph{patch}
 \[\mathcal{P}(0) := \{ p=(x,y,\theta) \: : \: |x(1 + \cos\theta) + y\sin\theta| < \lambda \}.\]
  \begin{figure}[!htbp]
 \centering
 \includegraphics[width=0.98\textwidth]{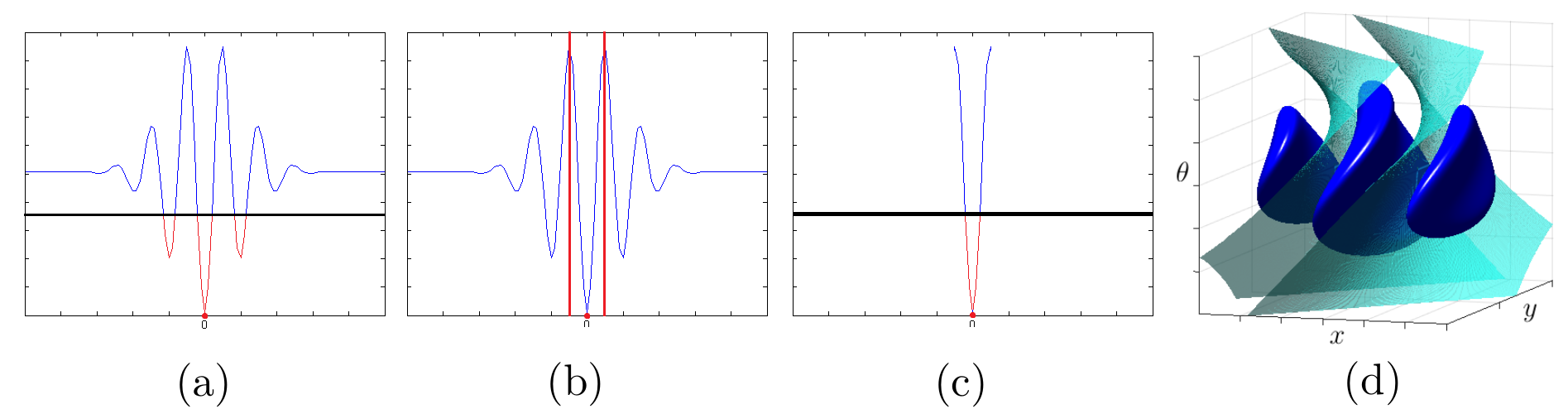}
 \caption{(a) For fixed $y=0$ and $\theta=0$, a plot of $x \mapsto d((x,0,0),(0,0,0))$. In red, the corresponding slice of a neighborhood $B_\varepsilon\big( (0,0,0) \big) = \{(x,y,\theta) \in \mathbb{R}^2\times S^1 \: : \:d((x,y,\theta),(0,0,0))<\varepsilon\}$, which is not connected. (b) We truncate the distance function at its maximum. (b) The neighborhood of the same radius as before, with the truncated distance, turns out to be connected. (d) The non-connected ball $B_\varepsilon\big( (0,0,0) \big)$ (dark blue) displayed in the 3D space $\mathbb{R}^2\times S^1$. The patch $\mathcal{P}(0,0,0)$ is represented by the volume between the two light blue surfaces. After truncating the distance function, only the central lobe remains. In this example we set $\lambda=1$ and $\sigma=1$. \label{dist_patch}}
 \end{figure}
 Of course, the thickness of the patch depends on the frequency of the oscillations of the distance, ruled by the wavelength $\lambda$ of the filters. Figure \ref{dist_patch}(a,b,c) schematically displays this operation on a plot of the distance function with respect to $x$, for fixed values of $y$ and $\theta$. The invariance of $d$ leads to the definition of a patch around each point $p_0=(x_0,y_0,\theta_0) \in \mathbb{R}^2\times S^1$ as follows.
 \[ \mathcal{P}(p_0) := \{ p=(x,y,\theta) \: : \: \big(R_{\theta_0}T_{(x_0,y_0)}(x,y),\: \theta-\theta_0\big) \in \mathcal{P}(0) \}. \]
 The shape of these patches is shown in Figure \ref{dist_patch}d.
 For each $p,p_0 \in \mathbb{R}^2\times S^1$ we then consider the distance $\tilde{d}(p,p_0)$ as defined in (\ref{globdist}).
 Note that those neighborhoods which are ``small enough'' are connected even without truncating the distance function (see Figure \ref{dist_patch}). In other words, there always exists an $\varepsilon>0$ such that $B_\varepsilon(p_0) \subseteq \mathcal{P}(p_0)$.
 This property, together with the symmetry of the patches, makes (\ref{globdist}) a global distance on $\mathbb{R}^2 \times S^1$ (see Proposition \ref{dtilde} and Remark \ref{remsymm}). Moreover, note that a finite sequence $\{q_j\}_{j=0,...,N}$ connecting two points always exists. For $p_0=(0,0,0)$ and $p=(x,y,\theta)$, take for example:
 \[q_0 = (0,0,0) = p_0, \quad q_1 = (0,y,0), \quad q_2 = \left(0,y,\frac{\pi}{2}\right), \quad q_3 = \left(x,y,\frac{\pi}{2}\right) \quad q_4 = (x,y,\theta) = p. \]
 The distance is therefore finite.

\paragraph{\textbf{Local estimate of $d^2$}}\label{gaborlocal}

 Let us study the \emph{local} behavior of the distance function $d$. Fix $p = (x,y,\theta) \in \mathbb{R}^2\times S^1$, and let $x,y,\theta \rightarrow 0$. We have:
 \begin{itemize}
 
 \item $\exp\left( -\frac{x^2}{4\sigma^2} -\frac{y^2}{4\sigma^2} - \frac{2\sigma^2\pi^2(1 - \cos\theta)}{\lambda^2} \right) \approx 1 - \frac{x^2}{4\sigma^2} -\frac{y^2}{4\sigma^2} - \frac{2\sigma^2\pi^2}{\lambda^2} \frac{\theta ^2}{2}$.
 
 \item $\cos\big( \pi \frac{(x(1 + \cos\theta) + y\sin\theta)}{\lambda} \big) \approx 1 - \frac{2\pi^2}{\lambda^2} x^2$.
 
 \end{itemize}
 Then
 \[
  d^2(p,0) \approx 2\sigma^2\pi \left(\left(\frac{1}{4\sigma^2}+\frac{2\pi^2}{\lambda^2}\right)x^2 +\frac{y^2}{4\sigma^2} + \frac{\sigma^2\pi^2}{\lambda^2}\theta^2\right).
 \]
More generally, for $p=(x,y,\theta) \rightarrow (x_0,y_0,\theta_0)=p_0$,
\begin{align*}
 d^2(p,p_0) \approx 2\sigma^2\pi \left(\left(\frac{1}{4\sigma^2}+\frac{2\pi^2}{\lambda^2}\right)a^2 +\frac{1}{4\sigma^2}b^2 + \frac{\sigma^2\pi^2}{\lambda^2}(\theta-\theta_0)^2\right),
\end{align*}
where $(a,b) = R_{\theta_0} T_{(x_0,y_0)}(x,y)$.
 Equivalently,
\begin{align*}
       d^2(p,p_0) \approx (x-x_0,\: y-y_0,\: \theta-\theta_0) \cdotp 
       g(p_0) \cdotp
       \begin{pmatrix} x-x_0 \\ y-y_0 \\ \theta-\theta_0 \end{pmatrix}
      \end{align*}
 where 
\begin{align*}
 g(p_0) = 2\sigma^2\pi \begin{pmatrix} \left(\frac{1}{4\sigma^2}+\frac{2\pi^2}{\lambda^2}\right)\cos^2\theta_0 + \frac{1}{4\sigma^2}\sin^2\theta_0 & \frac{2\pi^2}{\lambda^2}\cos\theta_0\sin\theta_0 & 0 \\
        \frac{2\pi^2}{\lambda^2}\cos\theta_0\sin\theta_0 & \left(\frac{1}{4\sigma^2}+\frac{2\pi^2}{\lambda^2}\right)\sin^2\theta_0 + \frac{1}{4\sigma^2}\cos^2\theta_0 & 0 \\
        0 & 0 & \frac{\sigma^2\pi^2}{\lambda^2}
       \end{pmatrix}.
      \end{align*}
Thus, the distance $d$ is locally equivalent to a Riemannian distance on $\mathbb{R}^2 \times S^1$.\\
Note that, for every point $p_0$,
\[\det g(p_0) = 8\sigma^6\pi^3 \left(\frac{1}{4\sigma^2}+\frac{2\pi^2}{\lambda^2}\right)\frac{1}{4\sigma^2} \frac{\sigma^2\pi^2}{\lambda^2}.\]
This implies that the associated Riemannian measure on $\mathbb{R}^2 \times S^1$ is a constant multiple of the Euclidean measure.

\paragraph{\textbf{Convergence to a sub-Riemannian metric}}\label{gaborconv}

 Finally, we show that the metric $g$ computed above is a Riemannian approximation to a sub-Riemannian structure on $\mathbb{R}^2\times S^1$ which is, up to constants, the same as the one defined in \cite{cs06}. More precisely, we let:
 \begin{enumerate}[(i)]
  \item $\sigma^2 = A\lambda$ for some $A>0$.
  \item $\lambda \longrightarrow 0$.
 \end{enumerate}
 This means that the support of the filters shrinks and the number of oscillations under the Gaussian bell goes to infinity.
 We have, for each $p_0 = (x_0,y_0,\theta_0)$,
 \begin{align*}
  g(p_0) = 2\lambda\pi \begin{pmatrix} \frac{1}{4A\lambda}+\frac{2\pi^2}{\lambda^2}\cos^2\theta_0 & \frac{2\pi^2}{\lambda^2}\cos\theta_0\sin\theta_0 & 0 \\ &&& \\
        \frac{2\pi^2}{\lambda^2}\cos\theta_0\sin\theta_0 &  \frac{1}{4A\lambda}+\frac{2\pi^2}{\lambda^2}\sin^2\theta_0 & 0 \\ &&& \\
        0 & 0 & \frac{A\pi^2}{\lambda}
       \end{pmatrix}.
      \end{align*}
 The inverse metric reads:
 \begin{align*}
  g^{-1}(p_0) = \frac{1}{2\pi\left(\frac{\pi^4}{4} + \lambda\frac{\pi^2}{16A}\right)} \begin{pmatrix} 2A\pi^4\sin^2\theta_0 + \lambda\frac{\pi^2}{4} & -2A\pi^4\sin\theta_0\cos\theta_0 & 0 \\ & & & \\ -2A\pi^4\sin\theta_0\cos\theta_0 & 2A\pi^4\cos^2\theta_0 + \lambda\frac{\pi^2}{4} & 0 \\ & & & \\ 0 & 0 & \frac{\pi^2}{2A} + \lambda\frac{1}{16A^2} \end{pmatrix}.
 \end{align*}
 Now set $\lambda \rightarrow 0$. We have:
 \begin{equation}\label{cometric} g^{-1}(p_0) \xrightarrow[\quad \lambda\rightarrow 0 \quad ] \quad \frac{4A}{\pi} \begin{pmatrix}
                     \sin^2\theta_0 & -\cos\theta_0 \sin\theta_0 & 0 \\
                     -\cos\theta_0 \sin\theta_0 & \cos^2\theta_0 & 0 \\
                     0 & 0 & \frac{1}{4\pi^2 A^2} \end{pmatrix} =: g_0^{-1}(p_0), \end{equation}
 i.e. the metric is the Riemannian approximation to a sub-Riemannian structure on $\mathcal{G} = \mathbb{R}^2\times S^1$. In particular, the matrix $g_0^{-1}(p_0)$ in (\ref{cometric}) is the cometric induced by the vector fields $Y_1$ and $Y_2$ defined in (\ref{vecY}), with the following norm on the horizontal planes:
 \begin{equation}\label{horiznorm} |v|^2 = \frac{\pi}{4A}|Y_1 \cdotp v|^2 +  A\pi^3|Y_2 \cdotp v|^2  \quad \forall v \in H_{x,y,\theta}.\end{equation} 
 For each $(x,y,\theta)$ the horizontal plane $H_{x,y,\theta}$ is the subspace of $T_{x,y,\theta}\mathcal{G}$ generated by $Y_1$ and $Y_2$. Note that $g_0^{-1}$ is just a notation, since this matrix is not invertible.

\subsection{A non-differential example}\label{nondiff}
 
 The example that we are about to introduce is a relevant one since it represents an instance of feature space whose metric cannot be described through a differential structure, thus motivating our work in the more general setting of metric spaces.\\
  Let us consider a surface
  \begin{equation}\label{sigma} \Sigma = \big\{ (x,y,\theta) \in \mathbb{R}^2\times S^1 \; : \; \theta = \Theta(x,y) \big\},\end{equation}
  and the corresponding subset $\{\psi_{x,y,\Theta(x,y)}\}_{x,y}$ of the above-mentioned family of Gabor filters. This yields a feature space $\mathcal{G}\approx\mathbb{R}^2$, endowed with the metric structure defined by this subfamily of filters, i.e.
 \begin{equation}
  d\big((x,y),(x_0,y_0)\big) := \|\psi_{x,y,\Theta(x,y)}-\psi_{x_0,y_0,\Theta(x_0,y_0)}\|_{L^2(\mathbb{R}^2)}.
 \end{equation}
 The restriction to $\Sigma$ of a distance which is locally equivalent to a Riemannian one is still locally equivalent to the induced Riemannian metric on the surface. However, setting $\sigma^2 = A\lambda \rightarrow 0$ as before yields a sub-Riemannian structure on $\mathbb{R}^2 \times S^1$, determined by the vector fields $Y_1$ and $Y_2$ of Eq. (\ref{vecY}). We can consider them to be rescaled so that the norm (\ref{horiznorm}) becomes the Euclidean norm on the horizontal planes. We still denote their commutator by $Y_3$.\\
 We work on the domain $(x,y)$ of the function $\Theta$ defining the surface. Now consider, as in \cite{cittimanfr}, the projections $V:=Y_{1\:\Theta}$ and $W:=Y_{3\:\Theta}$ of the vector fields $Y_1,Y_3$ on the plane $P = \{(x,y,0)\}$:
 \begin{align*}
  V_{(x,y)} = Y_{1\:\Theta\:(x,y)} = - \sin(\Theta(x,y)) \partial_x + \cos(\Theta(x,y)) \partial_y \\
  W_{(x,y)} = Y_{3\:\Theta\:(x,y)} = \cos(\Theta(x,y)) \partial_x + \sin(\Theta(x,y)) \partial_y.
 \end{align*}
 The vector fields $V$ and $W$ span the plane $P$. Note that the surface $\Sigma$ is foliated by integral curves of $V$, and the restriction of the horizontal norm (\ref{horiznorm}) onto this surface would yield a degenerate distance whose balls are segments of curves. \\
 The distance we want to consider on $\Sigma$ is instead the one whose balls are obtained by intersecting $\Sigma$ with the balls of the sub-Riemannian metric on $\mathbb{R}^2\times S^1$ -- i.e. the distance induced on $\Sigma$ as a metric subspace of $\mathbb{R}^2\times S^1$ with the Carnot-Carath\'{e}odory distance. At each point $p_0 \in \mathbb{R}^2 \times S^1$, the exponential mapping $\exp_{p_0} : \mathfrak{g} \rightarrow \mathbb{R}^2 \times S^1$ is defined by $\exp_{p_0}(X) = \gamma_X(1,p_0)$, where $\mathfrak{g}$ is the Lie algebra associated to $\mathbb{R}^2 \times S^1 = SE(2)$ as a Lie group (see \cite{warner}) and $\gamma_X(\cdot,p_0)$ is the unique solution to the Cauchy problem
 \begin{equation*}
  \begin{cases} \partial_t \gamma(t) = X_{|\gamma(t)}\\
   \gamma(0) = p_0.
  \end{cases}
 \end{equation*}
 For sufficiently small $t$, $\exp_{p_0}(tX) = \gamma_{tX}(1,p_0) = \gamma_X(t,p_0)$ is always well defined. Moreover, $\exp_{p_0}$ is a local diffeomorphism \cite{warner}. We can thus define \emph{locally} on $\mathbb{R}^2 \times S^1$ the distance
  \[d_Y^2(p,p_0) = v_1^2 + v_2^2 + |v_3|,\]
  where $v_1,v_2,v_3 \in \mathbb{R}$ are such that $p = \exp_{p_0}\left( \sum_{i=1}^3 v_i Y_i \right)$. This distance is locally equivalent to the Carnot-Carath\'{e}odory distance $d_{cc}$ on $\mathbb{R}^2 \times S^1$. Restricted on the domain of $\Theta$, this becomes
 \begin{equation}
  d^2_\Sigma\big( (x,y),(x_0,y_0) \big) = e_1^2 + |e_2|,
 \end{equation}
 where $(x,y) = \exp_{(x_0,y_0)}\big( e_1 V + e_2 W \big)$ (see \cite{cittimanfr}). Note that the balls of this distance are indeed open sets of the surface.\\ 

 \begin{figure}[htbp!]
  \centering
  \includegraphics[width=0.5\textwidth]{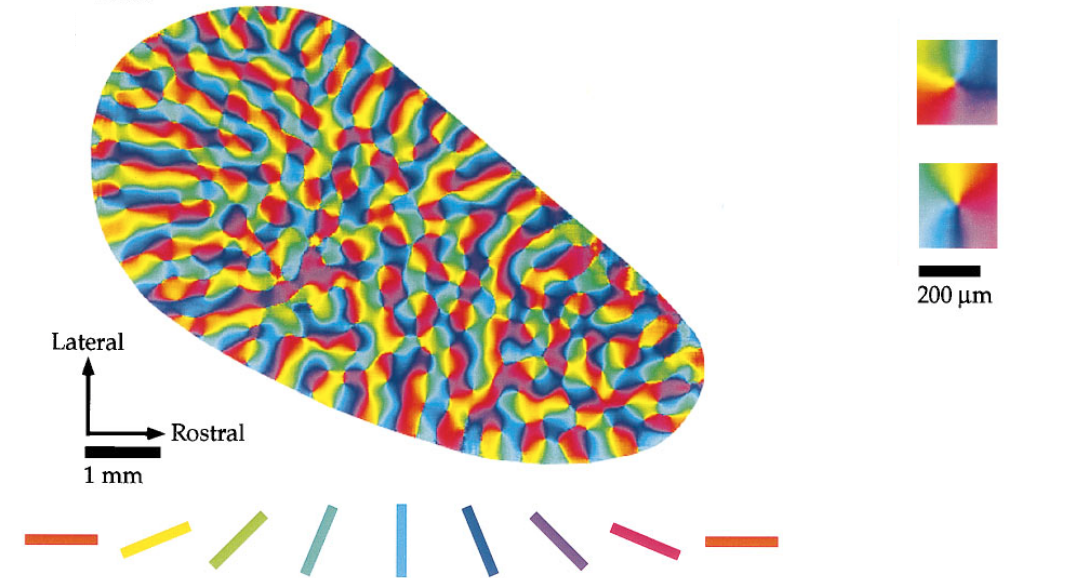}
  \caption{An orientation map $\Theta(x,y)$. The orientation preference measured at each location $(x,y)$ is color-coded. On the right, enlarged portions of the map show pinwheel arrangements, corresponding to hypercolumns. Image modified from \cite{bosking}. \label{pinw}}
 \end{figure}
 Surfaces play a key role in modeling the visual cortex. A first example is given by a surface of maxima such as the one introduced in Section \ref{subriemannian}. Another important instance is represented by the surface defined through an orientation map of V1. These maps, which can be computed through optical imaging techniques (see \cite{bosking}), express the fine-scale mapping of orientation preference of V1 neurons: the visual cortex is indeed two-dimensional and each hypercolumn actually consists of a \emph{pinwheel} configuration such as the ones displayed in Figure \ref{pinw}.

 We shall return to this example in the next Section, whose main subject will be the horizontal connectivity of V1. As already mentioned in Section \ref{subriemannian}, a possible way to represent this connectivity is by means of a diffusion process: in some differential models of V1, this diffusion is expressed through second order operators associated to the sub-Riemannian structure taken into consideration. In order to still be able to use this approach in non-differential cases such as the one described above, we aim at extending it to the context of metric measure spaces.

\section{Connectivity}\label{connectivity}

 A central aspect in modeling the visual cortex is the characterization of how the activity of a neuron is influenced by the surrounding cells.
 In the main existing mathematical cortex models, the feature space (obtained as the set of parameters indexing a family of filters) is equipped with a sub-Riemannian structure (\cite{scp}, \cite{cs06}). Starting from this local geometry, the idea is that of describing the spreading of horizontal connections around each neuron through a propagation equation (e.g. the sub-Riemannian heat equation or the Fokker-Planck equation) associated to the geometry of the space. Such constructions inspired us to give an analogous description of the lateral connectivity through a suitable concept of diffusion linked to the geometric structure of our space.\\
 In our model, the feature space is equipped with a metric space structure defined by the receptive profiles themselves. Starting from a general family of filters, we cannot expect the distance obtained to be compatible with some differential structure. We shall then address the issue in the much more general setting of \textit{metric measure spaces}, following the classical approach of Sturm (\cite{sturm95b},\cite{sturm98}) in defining Dirichlet forms and diffusion processes in this context: if one can define a measure satisfying a certain condition of compatibility with the distance, then it is possible to obtain well-behaved extensions of the Laplacian operator and of its heat kernel. The latter can also be shown to admit Gaussian estimates in terms of the distance.

 \subsection{Diffusion processes on metric measure spaces}\label{sturm}

 In this section we shall briefly summarize the content of \cite{sturm98}, in which Sturm provided a general method to construct a diffusion process on a metric measure space, and proved some properties under a crucial assumption, called the \textit{Measure Contraction Property}.

 Let $(X,d,\mu)$ be a metric measure space, such that $(X,d)$ is a locally compact separable metric space and $\mu$ is a Radon measure on $X$, strictly positive on nonempty open sets. One can then construct a Dirichlet form $E$ on $L^2(X,\mu)$ as the $\Gamma$-limit of a sequence of forms, defined in analogy with the Dirichlet form
 \begin{equation}
 u \longmapsto \frac{1}{2} \int_X |\nabla u |^2 d\mu,
 \end{equation}
 whose associated elliptic operator is the Laplace-Beltrami operator, in the Riemannian case. More precisely, one defines 
 \begin{equation}
 E^r(u) = \frac{1}{2} \int_X \mathcal{N}(x) \int_{B_r(x)\smallsetminus \{x\}} \left(\frac{u(z)-u(x)}{d(z,x)}\right)^2 \frac{d\mu(z)}{\sqrt{\mu(B_r(z))}} \frac{d\mu(x)}{\sqrt{\mu(B_r(x))}},
 \end{equation}
where $\mathcal{N}$ is a normalization function, and lets $E = \Gamma$-$\lim_{r\rightarrow 0} E^r$. This does always exist, provided that $(X,d,\mu)$ satisfies the following property.
\begin{defi}
A metric measure space $(X,d,\mu)$ satisfies the \textit{(weak) Measure Contraction Property} (MCP) \textit{with exceptional set} if there exists a closed set $Z \subseteq X$  with $\mu(Z) = 0$ such that for every compact set $Y \subseteq X \smallsetminus Z$ there are numbers $R>0, \Theta<\infty$ and $\vartheta<\infty$, and $\mu^2$-measurable maps $\Phi_t : X \times X \rightarrow X$ (for all $t \in [0,1]$), with the following properties.
\begin{enumerate}[(i)]
\item for $\mu$-a.e. $x,y \in Y$ with $d(x,y)<R$, and for all $s,t \in [0,1]$,
\begin{equation} \Phi_0(x,y) = x, \quad \Phi_t(x,y) = \Phi_{1-t}(y,x), \quad \Phi_s(x,\Phi_t(x,y)) = \Phi_{st}(x,y),\end{equation}
\begin{equation}\label{geo} d(\Phi_s(x,y), \Phi_t(x,y)) \leq \vartheta |s-t| d(x,y).\end{equation}
\item Define, for $r<0$, the measures $ d\mu_r(x) = \frac{d\mu(x)}{\sqrt{\mu(B_{r}(x))}}$.
Then, for all $r<R$, $\mu$-a.e. $x \in Y$, all $\mu$-measurable $A \subseteq B_r(x)\cap Y$ and all $t \in [0,1]$,
\begin{equation}\label{mcp} \frac{\mu_r(A)}{\sqrt{\mu(B_r(x))}} \leq \Theta \: \frac{\mu_{rt}(\Phi_t(x,A))}{\sqrt{\mu(B_{rt}(x))}}. \end{equation}
\end{enumerate}
The space $(X,d,\mu)$ is said to verify the \textit{strong MCP} if:
\begin{itemize}
 \item the constants $\Theta$ and $\vartheta$ can be taken arbitrarily close to 1;
 \item for every $\Theta'>1$ there exists some $\vartheta'>1$ such that, for $\mu$-a.e. $x \in Y$ and for all $r<R$ with $B_r(x) \subseteq Y$,
 $$ \mu(B_{r\vartheta'}(x)) \leq \Theta' \mu(B_r(x)). $$
\end{itemize}
In this case, there is no restriction in taking always $\Theta=\Theta'$ and $\vartheta<\vartheta'$. \\
For both the weak and strong MCP, one says \textit{without exceptional set} if $Z = \emptyset$.
\end{defi}
\begin{rem}
 For fixed $x$ and $y$, the map $\Phi_\cdotp (x,y) : [0,1] \rightarrow X, \; \; t \mapsto \Phi_t (x,y)$ is a \textit{quasi-geodesic} joining $x$ and $y$. Moreover, if $(X,d)$ is a geodesic space such that geodesics joining $x$ and $y$ can be chosen in such a way that they depend in a measurable way on $x$ and $y$, property (ii) simplifies to
\[ \frac{\mu(A)}{\mu(B_r(x))} \leq \Theta \: \frac{\mu(\Phi_t(x,A))}{\mu(B_{rt}(x))}. \]
\end{rem}
\begin{ex}
 Let $(X, g)$ be a Riemannian manifold. If $d$ is the Riemannian distance and $\mu$ is the Riemannian volume on X, $(X,d,\mu)$ is a metric measure space satisfying the initial requests on $d$ and $\mu$. Furthermore, for such a space the strong MCP without exceptional set is verified. This example is central in our setting, since the basic case which we have in mind as a prototype is the cortical metric space associated to Gabor filters, whose distance is locally equivalent to a Riemannian distance.\\
 More examples are given by manifolds with corners or by gluing together of manifolds not necessarily of the same dimension.
\end{ex}

The MCP implies some important facts, among which the volume doubling property. Moreover, for each $u \in C_0^{Lip}(X)$, the $\Gamma$-limit and the point-wise limit of $E^r(u)$ exist and coincide. Such a limit defines a strongly local, regular Dirichlet form, whose associated intrinsic metric is locally equivalent to the original distance $d$. A Poincar\'e inequality is shown to hold as well. Finally, the corresponding positive self-adjoint operator $A$ has a H\"{o}lder continuous heat kernel $h_t$ (see Theorem 7.4 in \cite{sturm98}):
 \begin{theo}
There exists a measurable function 
\begin{equation}
 H \; : \; ]0,\infty[ \times X \times X \longrightarrow [0,\infty], \quad (t,x,y) \longmapsto H(t,x,y) \equiv h_t(x,y)
\end{equation}
with the following properties.
\begin{enumerate}[(i)]
 \item For every $t>0$, every $u \in L^2(X,\mu)$ and $\mu$-a.e. $x \in X$,
 \begin{equation}\label{expA} e^{-At} u (x) = \int_X h_t(x,y) u(y) d\mu(y).\end{equation}
 \item The function $H$ is locally H\"{o}lder continuous on $]0,\infty[ \times (X\smallsetminus Z) \times (X\smallsetminus Z)$ and identically zero on its complement in $]0,\infty[ \times X \times X$.
 \item For all $s,t>0$ and all $x,y \in X$,
 \begin{equation} 
  h_t(x,y) = h_t(y,x) \quad \text{ and } \quad  h_{t+s}(x,y) = \int_x h_s(x,z)h_t(z,y)d\mu(z).
 \end{equation}
\end{enumerate}
The function $H$ is defined point-wise uniquely by these properties and is called \emph{heat kernel} for $A$.
\end{theo}
Furthermore, this heat kernel admits upper and lower Gaussian estimates. More precisely (see Theorems 7.7 and 7.9 of \cite{sturm98}), we have the following result.
\begin{theo}\label{heatgauss1}
Let $(X,d,\mu)$ verify the strong MCP with some exceptional set, and let $Z$ be the exceptional set for the weak MCP. Then, for every compact $Y \subseteq X \smallsetminus Z$ and every $\varepsilon>0$, there exists a constant $C$ such that
\begin{align*}
\frac{1}{C \: \mu(B_{\sqrt{t}\wedge R}(x))} \exp\left( -C \frac{d^2(x,y)}{2t} \right) \exp \left( - \frac{C t}{R^2} \right) \leq h_t(x,y) \\ \leq \frac{C}{\mu(B_{\sqrt{t_0}}(x))} \exp\left( -\frac{d^2(x,y)}{(2+\varepsilon)t} \right) \exp\left( -(1+\varepsilon)\Lambda t\right),
\end{align*}
for each $x,y$ which are joined by an arc $\gamma$ in $Y$ of arc length $d(x,y)$. Here $R = d(\gamma, X \smallsetminus Y)$, $t_0 = \inf\{t, d^2(x,X\smallsetminus Y), d^2(y,X\smallsetminus Y)\}$ and $\Lambda$ is the bottom of the spectrum of the operator $A$ on $L^2(X,\mu)$. 
\end{theo}

\subsection{The cortical metric measure space}\label{hausdorff}

 Let us recall our setting: we have a metric space $(\mathcal{G},d)$, where $\mathcal{G}$ is the feature space indexing a family $\{\psi_p\}_{p \in \mathcal{G}}$ of linear filters on the plane, and $d$ is the distance function of Definition \ref{defdist}.\\
 The first step in order to be able to do some analysis on $(\mathcal{G},d)$ is to equip it with a suitable measure. This has to be related to some notion of \textit{density} of the filters with respect to the distance $d$. Moreover, the MCP of \cite{sturm98} expresses a link between the metric balls and the measure. Therefore, a quite natural choice is the spherical Hausdorff measure (see \cite{hausd}, \cite{yeh}, \cite{federer}) associated to the distance $d$. We shall denote it by $\mu$. Suppose now that $(\mathcal{G},d)$ is a locally compact separable metric space, that $\mu$ is a Radon measure with full support on $X$ and that the metric measure space $(X,d,\mu)$ satisfies the MCP. This yields Gaussian estimates for the heat kernel $h_t$ associated to the diffusion process defined by the Dirichlet form $E$, meaning that in this case one has an approximate version of $h_t$ expressed \textit{explicitly} in terms of the cortical distance $d$.

 The first comment to be made is about the nice behavior of the spherical Hausdorff measure and of the MCP in the event of equivalent (or locally equivalent) distances. Indeed, if $d$ and $d'$ are two distances defined on $X$ with corresponding spherical Hausdorff measures $\mu$ and $\mu'$, then we have \cite{federer}:
 \begin{align} \exists \kappa>0 \; : \; & \kappa^{-1} d(x,y)  \leq d'(x,y) \leq \kappa \: d(x,y) \quad \forall x,y \in X \label{disteq}\\ \Longrightarrow & \quad \kappa^{-s} \mu(A) \leq \mu'(A) \leq \kappa^s \: \mu(A) \quad \text{ for any Borel set } A \subseteq X, \label{measeq}\end{align}
 where $s$ is the Hausdorff dimension.
 The same holds locally if the distances are only locally equivalent. 
 Suppose now that the MCP is verified for $(X,d, \mu)$. The exceptional set $Z$ is obviously still a null set with respect to $\mu'$. Fix a compact set $Y \subseteq X$. Note that, even if the equivalence is just local, the compactness of $Y$ allows to have (\ref{disteq}) with the same $\kappa$ over all $Y$. The maps $\Phi_t$ are still measurable and verify the properties (i). In particular (\ref{geo}) holds thanks to the equivalence of the distances. Recall that the doubling property holds for $d,\mu$ and let $M$ be a doubling constant:
 \[ \mu(B_{\kappa r}(x)) \leq M \mu(B_r (x)). \]
 Now, (\ref{measeq}) implies:
 \[ (\kappa^s M)^{-1} \mu(B_r(x)) \leq \mu'(B'_r(x)) \leq \kappa^s M \: \mu(B_r(x)) \quad \text{and} \quad 
 (\kappa^{\frac{3s}{2}} M )^{-1} \mu_r(A) \leq \mu'_r(A) \leq \kappa^{\frac{3s}{2}} M \: \mu_r(A).\]
 Finally, by these inequalities and the property (\ref{mcp}) for $d$ and $\mu$, we have:
 \begin{align*}
 \frac{\mu'_r(A)}{ \sqrt{\mu'(B'_r(x))}} \leq \frac{\kappa^{2s} M^\frac{3}{2} \: \mu_r(A)}{\sqrt{\mu(B_r(x))}} \leq \frac{\Theta \kappa^{2s} M^\frac{3}{2} \:\mu_{rt}(\Phi_t(x,A))}{\sqrt{\mu(B_{rt}(x))}} \leq \frac{\Theta \kappa^{4s} M^3 \mu'_{rt}(\Phi_t(x,A))}{\sqrt{\mu'(B'_{rt}(x))}},
 \end{align*}
 i.e.
 \[ \frac{\mu'_r(A)}{\sqrt{\mu'(B'_r(x))}} \leq C \: \frac{\mu'_{rt}(\Phi_t(x,A))}{\sqrt{\mu'(B'_{rt}(x))}}.\]
 To sum up, 
 \begin{propo} The weak MCP for the spherical Hausdorff measure is invariant under local equivalence of distances.\end{propo}
 In fact, if the equivalence constant $\kappa$ of the two distances can be locally chosen to be arbitrarily close to 1, then even the strong MCP is preserved. \\
 These facts are of particular importance for our purposes, since we have seen that the cortical distance arising from a set of Gabor filters turns out to be locally equivalent to a Riemannian distance on $\mathbb{R}^2\times S^1$, with a local equivalence constant approaching 1. Recall that all Riemannian manifolds $(M,g)$ are locally compact as metric spaces with the geodesic distance $d_g$ \cite{abb}, and that this property is preserved in metric spaces under local equivalence of distances. Moreover, the MCP holds \cite{sturm98} on $(M,d_g,\mu_g)$ where $\mu_g$ is the Riemannian measure, which coincides with the spherical Hausdorff measure associated to $d_g$. This immediately leads to the following result.
 \begin{theo} The cortical metric measure space $(\mathbb{R}^2\times S^1,d,\mu)$ defined by the bank of Gabor filters (\ref{mother}) satisfies the MCP. \end{theo}

\subsection{The MCP for a sub-Riemannian surface in $\mathbb{R}^2 \times S^1$}

 We now go back to the example, introduced in Section \ref{nondiff}, of a sub-Riemannian surface in $\mathbb{R}^2 \times S^1$. We prove that such a space satisfies the MCP, thus providing an example of a non-differential feature space on which the horizontal connectivity can still be represented through a suitable diffusion process.
 
 Consider a surface $\Sigma$ as in (\ref{sigma}), whose defining function $\Theta(x,y)$ is $C^1$, except possibly for a discrete set $\varPi \subset \mathbb{R}^2$. Denote  $Z := \{(x,y,\Theta(x,y)) : (x,y) \in \varPi\} \subseteq \Sigma$.
 \begin{rem}
 $(\Sigma,d_\Sigma)$ is locally compact. Indeed,
  \begin{enumerate}[(i)]
   \item $\mathbb{R}^2 \times S^1$ with the Carnot-Carath\'{e}odory distance $d_{cc}$ is a locally compact space \cite{abb}. Then each closed subset of $(\Sigma,d_\Sigma)$ away from $Z$ is locally compact because it is a closed subspace of $(\mathbb{R}^2 \times S^1,d_{cc})$ by the continuity of $\Theta$.
   \item Now, given $\zeta \in Z$, we need to construct a compact neighborhood of $\zeta$ in $\Sigma$. Consider a closed ball $\overline{B_\varepsilon^{cc}(\zeta)}$ of $d_{cc}$ in $\mathbb{R}^2 \times S^1$ such that $\overline{B_\varepsilon^{cc}(\zeta)}$ does not contain any other point of $Z$; then define $B := \overline{B_\varepsilon^{cc}(\zeta)} \cap \Sigma$. This is a neighborhood of $\zeta$ in the induced metric $d_\Sigma$. We now prove that $B$ is compact.\\
  Given a sequence $\{p_n\}_n \subseteq B \subseteq \overline{B_\varepsilon^{cc}(\zeta)}$, by the compactness of $\overline{B_\varepsilon^{cc}(\zeta)}$ there exists a subsequence $\{p_{n_k}\}_k$ converging to a point $p \in \overline{B_\varepsilon^{cc}(\zeta)}$. If $\zeta \notin \{p_{n_k}\}_k$, then $p$ belongs to $B$ by (i). If $\zeta \in \{p_{n_k}\}_k$, either $p=\zeta \in B$ or $p_{n_k} \neq \zeta$ for $k>\overline{k}$ and the truncated sequence falls into the preceding case.
  \end{enumerate}
 \end{rem}
 The measure $\mu$ that we consider on $\Sigma$ is the one given by the sub-Riemannian area, since this coincides up to a constant with the spherical Hausdorff measure on $(\Sigma,d_\Sigma)$ (see \cite{galri},\cite{rectif}). Specifically, given a subset $S \subseteq \Sigma$,
 \begin{equation}
  \mu(S) = \int_S |N_h| d\Sigma.
 \end{equation}
 Here, $N_h$ is the orthogonal projection of a unit vector field normal to $\Sigma$ onto the horizontal distribution, and $d\Sigma$ is the Riemannian measure of $\Sigma$ induced by the projected vector fields $V$ and $W$ defined in Section \ref{nondiff}, i.e. 
 \[d\Sigma(x,y) = \sqrt{\det g_\Sigma(x,y)}dxdy.\]
 Now denote $\xi = (x,y)$. We define
 \begin{equation}\Phi_t(\xi,A) = \exp_\xi(t \cdotp \exp_\xi^{-1}(A)),\end{equation}
 where $\cdotp$ denotes the dilation 
  \[ t \cdotp v = \left(te_1,\: t^2 e_2\right) \quad \forall v=(e_1,e_2) \in P.\] 
 Given a compact set $Y$ of $\Sigma$ and $A \subseteq B_r(\xi)\cap Y$, we have:
 \begin{align*}
 \mu(\Phi_t(\xi,A)) & = \int_{\exp_\xi(t \cdotp \exp_\xi^{-1}(A))} d\mu = \int_{\exp_\xi(t \cdotp \exp_\xi^{-1}(A))} |N_h(\xi')| \sqrt{\det(g_\Sigma(\xi'))} d\xi' \\
 & = \int_{t \cdotp \exp_\xi^{-1}(A)} \: \big|J_\xi(v)\big| |N_h(\exp_\xi(v))| \sqrt{\det(g_\Sigma(\exp_\xi(v)))} dv \\
 & = t^3 \int_{\exp_\xi^{-1}(A)} \: \big|J_\xi(tu)\big| |N_h(\exp_\xi(tu))| \sqrt{\det(g_\Sigma(\exp_\xi(tu)))} du \\
 & = t^3 \int_A \frac{\big|J_\xi(t \cdotp \exp_\xi^{-1}(\xi'))\big|}{\big|J_\xi(\exp_\xi^{-1}(\xi'))\big|}|N_h(\Phi_t(\xi,\xi'))|\sqrt{\det(g_\Sigma(\Phi_t(\xi,\xi')))}d\xi',
 \end{align*}
 where $d\xi'= dx'dy'$ and $J_\xi(v)$ denotes the Jacobian determinant of $\exp_\xi$. Then
 \begin{equation}\label{ratio} \frac{\mu(\Phi_t(\xi,A))}{\mu(A)} = \frac{\int_A \big|J_\xi(t \cdotp \exp_\xi^{-1}(\xi'))\big| \: \big|J_\xi(\exp_\xi^{-1}(\xi'))\big|^{-1} f(\Phi_t(\xi,\xi'))d\xi'}{\int_A f(\xi')d\xi'} \; t^3 ,\end{equation}
 where we have denoted $f = |N_h| \sqrt{\det(g_\Sigma)}$. Now, one has (\cite{cittimanfr},\cite{nsw}):
 \[(1 + O(|v|))^{-1} \leq |J_\xi(v)| \leq 1 + O(|v|).\]
 This yields
 \begin{align*} \frac{\big|J_\xi(t \cdotp \exp_\xi^{-1}(\xi'))\big|}{\big|J_\xi(\exp_\xi^{-1}(\xi'))\big|} \geq (1+O(|\exp_\xi^{-1}(\xi')|))^{-1}(1+t^2O(|\exp_\xi^{-1}(\xi')|))^{-1} \geq 1+O(r).\end{align*}
 Therefore the initial calculation leads to
 \[\frac{\mu(\Phi_t(\xi,A))}{\mu(A)} \geq \frac{\int_A f(\Phi_t(\xi,\xi'))d\xi'}{\int_A f(\xi')d\xi'}(1 + O(r))\:t^3. \]
 Finally, $f(\Phi_t(\xi,\xi')) = f(\xi') + O(d_\Sigma(\Phi_t(\xi,\xi'), \xi'))$ and both $\Phi_t(\xi,\xi')$ and $\xi'$ are in $B_r(\xi)$. Then,
 \[\frac{\int_A f(\Phi_t(\xi,\xi'))d\xi'}{\int_A f(\xi')d\xi'} = \frac{\int_A f(\xi')d\xi' + O(r)\int_A d\xi'}{\int_A f(\xi')d\xi'} = 1 + O(r) \quad \Rightarrow \quad \frac{\mu(\Phi_t(\xi,A))}{\mu(A)} \geq (1 + O(r))\:t^3.\]
 On the other hand, since $B_{rt}(\xi) = \Phi_t(\xi,B_r(\xi))$ and since we have estimates for $J_\xi$ and $f$ both from above and from below, we have:
 \[ \frac{\mu(B_{rt}(\xi))}{\mu(B_r(\xi))} \leq (1 + O(r)) \; \frac{\mu(\Phi_t(\xi,A))}{\mu(A)}.\] 
 Note that, in the case of a pinwheel surface (see Section \ref{nondiff}), the \emph{exceptional set} $Z$ of the MCP is represented by the singularities at the center of each pinwheel arrangement.

 \subsection{Propagation through a connectivity kernel}\label{kernel_iter}

 Under the hypothesis that the MCP holds on $(X,d,\mu)$, the heat kernel admits Gaussian estimates. Hence, in analogy with the sub-Riemannian case, we may model the horizontal connectivity with the heat kernel. However, it would not be clear how this is implemented in the visual cortex. On the other side, the kernel $K$ can be locally approximated through an exponentially decaying function of the squared distance.
 \begin{rem}\label{taylor}
  By Taylor expansion, we have: $e^{-\frac{d^2(p,q)}{2}} =  \left(1 - \frac{d^2(p,q)}{2}\right) + o\left( d^2(p,q) \right)$.
  Recall that $d^2(p,q) = 2 (t - K(p,q))$, where $t$ is the squared $L^2$ norm of the filters. We shall now make explicit the dependence of $K$ and $d$ on $t$, i.e.
\[ K_t(p,q) = t - \frac{d_t^2(p,q)}{2}.\]
 If we fix $t$, when $d_t(p,q)$ is small we then get
 \begin{equation}\label{estimate12}
  e^{-\frac{d_t^2(p,q)}{2}} \approx \left(1 - \frac{d_t^2(p,q)}{2}\right) = K_t(p,q).
 \end{equation}
 \end{rem}
 
 In addition, in the special case of a compact Riemannian submanifold of the Euclidean space, an arbitrary good approximation of the heat kernel can be provided by iterating the Gaussian kernel for small $t$.
 In \cite{coiflaf}, an approximating kernel is defined as follows: given an exponentially decaying function $h$ and $\alpha \in \mathbb{R}$,
 \[k_t(p,q) = h\left( \frac{\|p-q\|^2}{t} \right); \quad k_t^{(\alpha)}(p,q) = \frac{k_t(p,q)}{Q_t^\alpha(p)Q_t^\alpha(q)}, \;\text{where } Q_t(p) = \int k_t(p,q)Q(q)d\mu(q).\]
 Here, $Q$ is a density function expressing the distribution of points in a dataset. 
 A new kernel, depending on the choice of $\alpha$, is then defined via a normalization:
 \begin{equation}\label{normalizationCL}
 S^{(\alpha)}[k_t](p,q) = \frac{k_t^{(\alpha)}(p,q)}{\int k_t^{(\alpha)}(p,q')Q(q')d\mu(q')}.
 \end{equation}
 The following theorem is proved (see Proposition 3 in \cite{coiflaf}). 
 \begin{theo}  Define the operators 
 \begin{equation}\label{apprheat} H^{(\alpha)}_t f \: (p) := \int S^{(\alpha)}([k_t](p,q)f(q)d\mu(q)
 \quad \text{and}\quad L_{t,\alpha}f = \frac{1}{t}\Big(f - H^{(\alpha)}_t f\Big)
 \end{equation}
 For $\alpha=1$  and for any fixed $N$, the operator $L_{t,1}$ converges to the Laplace-Beltrami operator onto the linear span of its first $N$ eigenfunctions, and the kernel 
 \[k_{t,n} := \left( H_\frac{t}{n}\right)^{n-1} S^{(\alpha)}[k_\frac{t}{n}]\]
 converges to the Neumann heat kernel on the manifold as $n$ goes to infinity.
 \end{theo}
 As a matter of fact, the result proved in \cite{coiflaf} is more general. For each value of $\alpha$, the generator converges to a specific operator (see Theorem 2 in \cite{coiflaf}). In particular, an interesting fact is that, for $\alpha=\frac{1}{2}$, the process approximates the diffusion of a Fokker-Planck equation depending on the density function $Q$. This result implies that different normalizations of the same Gaussian kernel may be used to define a generalization of other diffusion processes proposed in differential cortex models. As already mentioned in Section \ref{subriemannian}, the Fokker-Planck equation has been taken into consideration in various works to describe the lateral connectivity of V1 (\cite{mumford}, \cite{edge-stat}).

 We can interpret the association field generated by the lateral connectivity as the expansion of the activity starting from the stimulation of one specific profile (i.e. one point $p_0$ of the feature space $\mathcal{G}$), and we may model it through an operator analogous to (\ref{apprheat}).\\
 We adapt the normalization operation proposed in \cite{coiflaf} to our setting by taking the integrals w.r.t. the spherical Hausdorff measure associated to the cortical distance. For $\alpha=1$ and $Q \equiv 1$, we obtain the operator $S$ applied to kernels $\mathcal{K} : \mathcal{G}\times\mathcal{G} \rightarrow \mathbb{R}$ as follows:
 \[S[\mathcal{K}](p,q) = \frac{\mathcal{K}^{(1)}(p,q)}{\int \mathcal{K}^{(1)}(p,q')d\mu(q')}, \text{ where } \mathcal{K}^{(1)}(p,q) = \frac{\mathcal{K}(p,q)}{\int \mathcal{K}(p,q')d\mu(q') \; \int \mathcal{K}(p',q)d\mu(p')}.\]
 We then define the propagation operator
 \begin{equation} H_t f \: (p) := \int S\left[s\left(K_t\right)\right](p,q) \; f(q) \: d\mu(q),\end{equation}
 where $S$ is applied to the kernel $K_t$ divided by the norm $t$ of the filters and passed through a sigmoidal activation function
 \[s(z) = \frac{1}{1 + \exp(-z)}.\]
 Note that, for $d_t(p,q) \rightarrow +\infty$, the term $s\left(K_t\right)$ is an exponentially decaying function of $\frac{d_t^2(p,q)}{2}$:
 \[s\left(K_t(p,q)\right) = \frac{\exp\left(-\frac{d_t^2(p,q)}{2}\right)}{\frac{1}{e} + \exp\left(-\frac{d_t^2(p,q)}{2}\right)} \sim e \cdot \exp\left(-\frac{d_t^2(p,q)}{2}\right).\]
 We finally provide a description of the propagation of neural activity around a point $p_0$ by defining
 \begin{equation}\label{propK} K^{p_0}_{t,n} :=  H_t^{n-1} K^{p_0}_t,\end{equation}
 where $K^{p_0}_t(p) \equiv K^{p_0}_{t,1}(p) := S\left[s\left(K_t\right)\right](p,p_0)$.\\
 
 Remark \ref{taylor} suggests that, with a much more rough approximation, one can even think of modeling the cortical connectivity as an iteration of (a proper normalization of) $K_t$ itself instead of a Gaussian kernel. In this case, one may consider an activation function of the type $s(z) = \max(z-T, 0)$, which simply puts to zero all values below a certain threshold $T$.

\subsection{Numerical simulations and discussion}\label{simulations}

 We now present numerical simulations of the modeled propagation, both through the mechanism of repeated integrations described in Section \ref{kernel_iter}, and by suitably approximating the diffusion process introduced in Section \ref{sturm}. We first take into account the family of Gabor filters (\ref{mother}); we then consider the example of the bi-dimensional feature space induced by an orientation map, as in Section \ref{nondiff}.
 \paragraph{\textbf{Gabor filters}}
 In the Gabor case, the connectivity kernel lives in $\mathbb{R}^2 \times S^1$. Figure \ref{KconvK}b displays the projection onto the retinal plane of the kernel around the starting point $(0,0,0)$, obtained after four steps of the iterative rule (\ref{propK}). Figure \ref{KconvK}a displays the real part of the filter $\psi_{(0,0,0)}$, corresponding to the starting point. The propagation has been implemented on $]-1,5,1.5[\times ]-2,2[ \times ]-1.5,1.5[ \subseteq \mathbb{R}^2 \times S^1$, discretized with step 0.075 in $x$ and $y$ and with step 0.15 in $\theta$. We refer to our work \cite{neuro} for more technical details. The function 
 \begin{equation}\label{Kn} (x,y,\theta) \mapsto K_4^{(0,0,0)}(x,y,\theta) \end{equation}
 \begin{figure}[htbp!]
  \centering
  \includegraphics[width=.65\textwidth]{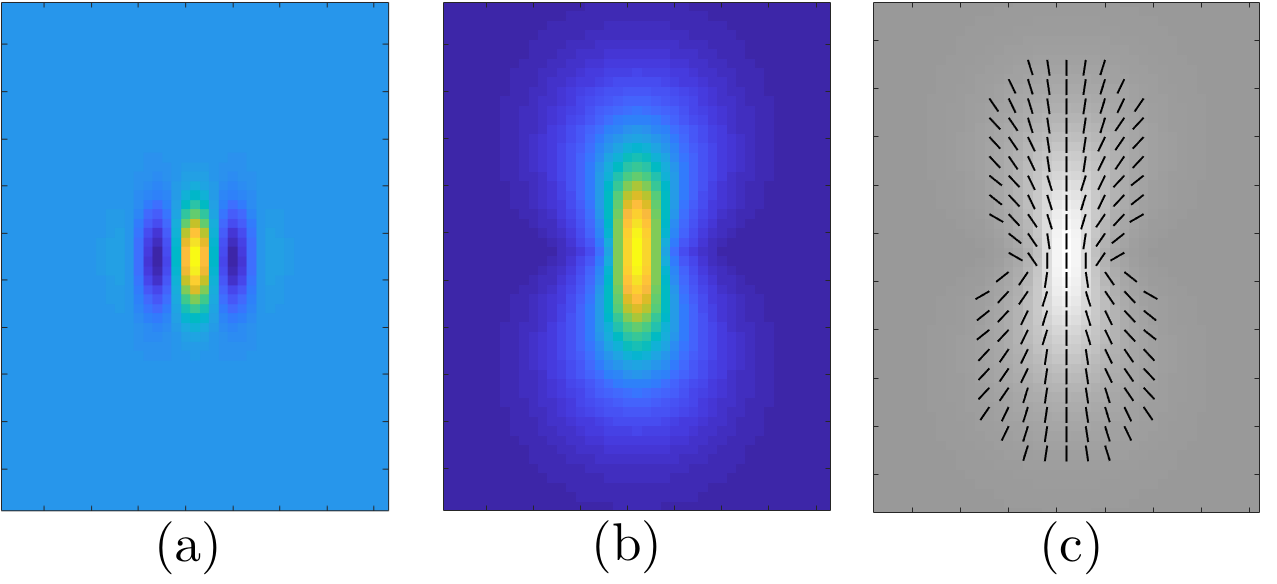}
  \caption{Propagation of the neural activity in $\mathbb{R}^2 \times S^1$ through repeated integrations of the kernel. (a) The starting filter $\psi_{(0,0,0)}$ (real part). (b) The kernel $K_4^{(0,0,0)}(x,y,\theta)$ obtained at the fourth step of propagation, projected down onto the $(x,y)$ plane by taking the maximum over $\theta$. (c) The corresponding maximizing orientations $\overline{\theta}(x,y)$: at every location $(x,y)$, an oriented segment with angle $\overline{\theta}(x,y)$ is displayed -- only where $K_4^{(0,0,0)}\big(x,y,\overline{\theta}(x,y)\big)$ is over a threshold. \label{KconvK}}
\end{figure}
 has been projected onto the $(x,y)$ plane by taking the maximum over the variable $\theta$. Figure \ref{KconvK}c shows the orientation $\overline{\theta}(x,y)$ maximizing the value of the kernel, at each location $(x,y)$ where this value exceeds a threshold. Specifically,
 \[ \overline{\theta}(x,y) := \arg\max_\theta K_4^{(0,0,0)}(x,y,\theta). \]
 Now, as shown in Section \ref{gaborcase}, the cortical distance $d$ obtained from the family of filters (\ref{mother}) is locally equivalent to a Riemannian distance on the space $\mathbb{R}^2\times S^1$. In such a case, it is possible to discretize the Laplace-Beltrami operator by means of a \emph{graph Laplacian operator} associated to the distance. Specifically, given a simple undirected weighted graph $\Gamma$ with vertices $X=\{p_i\}_i$ equipped with weights $\{\mu_i\}_i$, and edges $E=\{e_{ij}\}_{i,j}$ equipped with weights $\{w_{ij}\}_{i,j}$, one can define for any function $f$ on $V$ the Laplacian operator onto the graph as
 \begin{equation}\label{glap} Lf (p_i) := \frac{1}{\mu_i} \sum_{j \: : \: p_i\sim p_j} w_{ij} \left( f(p_j) - f(p_j) \right),\end{equation}
 where $p_i\sim p_j$ means that there is an edge connecting $p_i$ and $p_j$.
 This operator, possibly with slightly different definitions from time to time, is widely used in shape analysis (see e.g. \cite{graph}, \cite{graphsegm}) to construct algorithms that keep trace of the geometry of the data, by means of parameterizations obtained through the eigenfunctions of $L$.
 In \cite{graphappr}, a graph approximation of a Riemannian manifold $M$ is constructed by taking the set of vertices $X$ to be an $\varepsilon$-net in $M$ with an associated discrete measure $\tilde{\mu} = \sum_i \mu_i \delta_{p_i}$ which approximates the volume $\mu$ of $M$. In the Gabor case, this allows to consider a simple rectangular grid as set of vertices, provided that the discretization step is sufficiently small. Moreover, this choice yields uniform weights $\mu_i$. The set of edges with relative weights is then defined depending on the distance. Namely, for $\rho \gg \varepsilon$, two vertices $p_i,p_j \in X$ are connected by an edge iff $d_{ij} \equiv d(p_i,p_j)<\rho$, and in this case one defines the edge weight $\quad w_{ij} := \kappa \: \mu_i \mu_j$, where $\kappa$ is a normalization constant depending on the dimension of the manifold. Note that the vertices can be chosen to be any $\varepsilon$-net, since the geometry of the manifold is encoded in the definition of the edges, i.e. in the choice of the neighborhood over which the sum (\ref{glap}) is taken. We implemented the graph Laplacian associated to this approximating graph, in order to obtain a discretized heat equation on the same sampling of $\mathbb{R}^2\times S^1$ as before, with initial datum the Dirac delta $f_0 = \delta_{(0,0,0)}$ in this three-dimensional space, see Figure \ref{diff_graph}a. We took 100 iterations of the discretized differential equation with a time step of 0.01. We then projected the updated datum $f(x,y,\theta)$ onto the image plane, again by taking the maximum over $\theta$, and displayed the maximizing orientations as in the preceding case. See Figure \ref{diff_graph}b-c.\\
 The results obtained are compatible with the geometrical properties of V1 lateral connections, and the pattern of the maximizing orientations turns out to be consistent with the perceptual principles of association fields.
 \begin{figure}[htbp!]
  \centering
  \includegraphics[width=.9\textwidth]{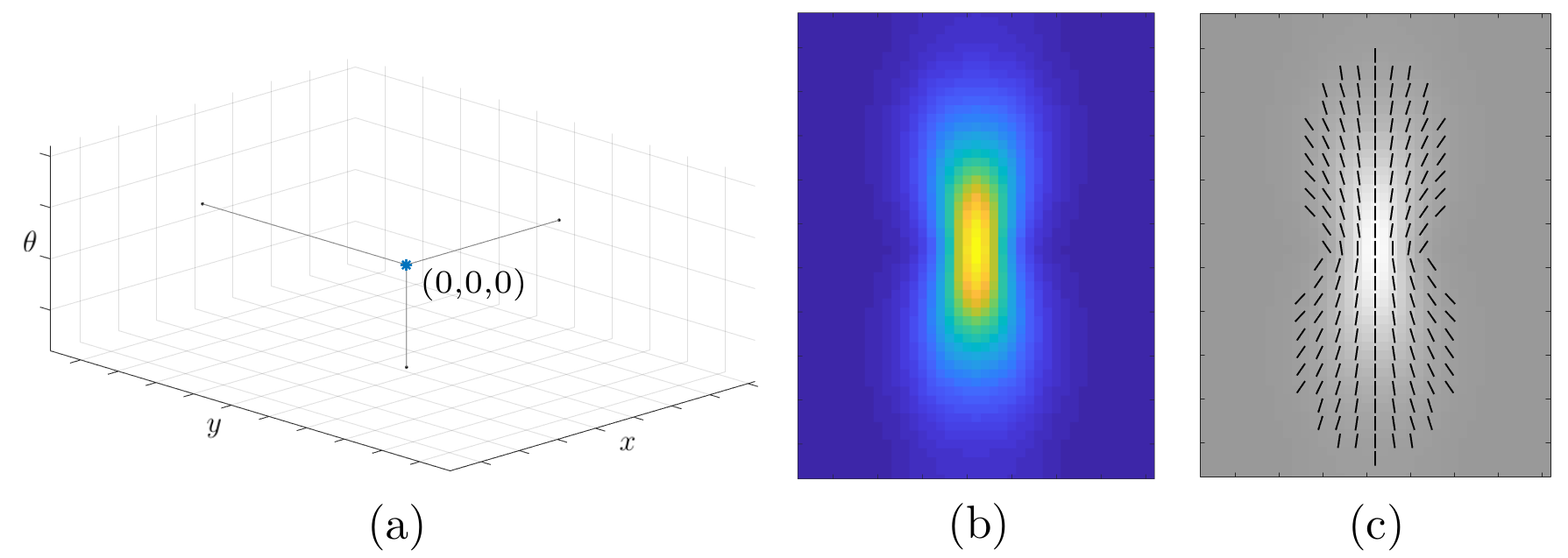}
  \caption{Propagation of the neural activity through the discretized heat equation associated to the graph Laplacian in $\mathbb{R}^2 \times S^1$. (a) The starting point $(0,0,0)$, displayed as a blue asterisk in this 3D space. (b) The updated kernel, projected down onto the $(x,y)$ plane by taking the maximum over $\theta$. (c) The corresponding maximizing orientations $\overline{\theta}(x,y)$, as in Figure \ref{KconvK}.}\label{diff_graph}
 \end{figure}
 
 \paragraph{\textbf{Orientation map}}
 We now consider the sub-family of Gabor filters $\{\tilde{\psi}_{x,y}\}_{x,y}$ defined by an orientation map $\Theta$ through $\tilde{\psi}_{x,y}=\psi_{x,y,\Theta(x,y)}$, and the corresponding metric structure onto $\mathcal{G}_\Theta = \mathbb{R}^2$. We generated an orientation map $\Theta$ through superposition of plane waves with random phases, as described in \cite{petitot}, and we chose its central point $(0,0)$ as a starting point: see Figure \ref{pw}(left). The corresponding filter $\tilde{\psi}_{0,0}$ is displayed in Figure \ref{pw}(right). Note that its orientation $\Theta(0,0)$ is determined by the chosen orientation map.
 \begin{figure}[htbp!]
  \centering
  \includegraphics[width=.5\textwidth]{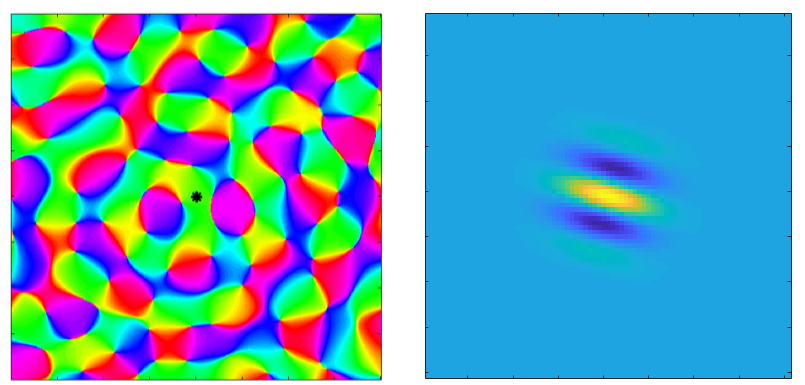}
  \caption{Left: the orientation map $\Theta$, generated through superposition of plane waves. The point $(0,0)$ is highlighted in black. Right: the starting filter $\tilde{\psi}_{0,0}$ (real part). \label{pw}}
 \end{figure}
 Again, we implemented the propagation of neural activity through iteration of the kernel onto the 2D feature space, and we displayed the updated kernel with color-coded intensity (Figure \ref{Kpw}a), as well as the orientations $\Theta(x,y)$ corresponding to points $(x,y)$ where the kernel exceeds a threshold (Figure \ref{Kpw}b). As a sampling of the feature space, we took $]-2,2[\times ]-2,2[ \subseteq \mathbb{R}^2=\mathcal{G}_\Theta$, discretized with step 0.05 for both $x$ and $y$. Finally, recall that the cortical distance on $\mathcal{G}_\Theta$ can be seen as the restriction of the Gabor distance to $\Sigma = \{\big(x,y,\Theta(x,y)\big)\}_{x,y} \subseteq \mathbb{R}^2 \times S^1$. This is still locally equivalent to a Riemannian metric on $\mathcal{G}_\Theta$. Note that, for $\sigma^2 = A\lambda \ll 1$, this approximates the distance $d_\Sigma$ of Section \ref{nondiff}. We implemented the graph Laplacian operator associated to this metric on $\mathcal{G}_\Theta$, and the corresponding discretized heat equation with initial datum $\delta_{(0,0)}$. The results for 150 iterations of the discretized equation, with a time step of 0.01, are displayed in Figure \ref{Lpw}. Note that in this case we do not need to project the connectivity kernels onto the image plane to visualize them, since the whole propagation already lives in a bidimensional space.\\
 \begin{figure}[htbp!]
  \centering
  \includegraphics[width=.6\textwidth]{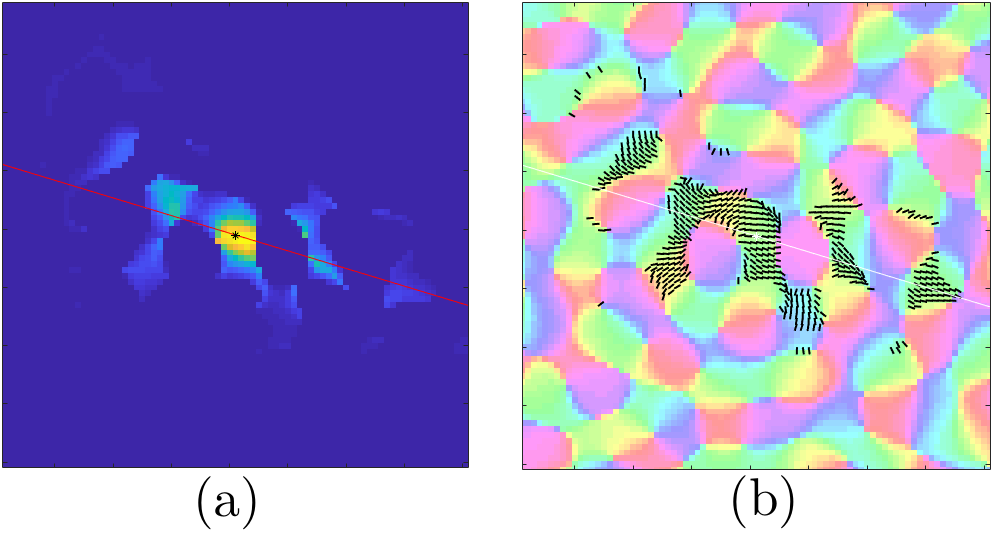}
  \caption{Propagation of the neural activity onto $\mathcal{G}_\Theta$ through repeated integrations of the kernel. (a) The propagated kernel around $(0,0)$, obtained after four iterations. A black asterisk shows the starting point $(0,0)$, and the orientation $\Theta(0,0)$ is highlighted by a red line superposed onto the image. (b) At each point $(x,y)$ where the kernel exceeds a threshold, the corresponding orientation $\Theta(x,y)$ is displayed through an oriented segment superposed onto the orientation map. \label{Kpw}}
 \end{figure}
 Again, the computed kernel spreads along the axis of the orientation $\Theta(0,0)$; moreover, it propagates in a \emph{patchy} way, with peaks in the regions of the map whose orientation values are close to $\Theta(0,0)$. This behavior has been observed experimentally by tracking the spreading of neural activity through biocytin injections, and by comparing it with the underlying orientation preference map \cite{bosking}.\\
 \begin{figure}[htbp!]
  \centering
  \includegraphics[width=.6\textwidth]{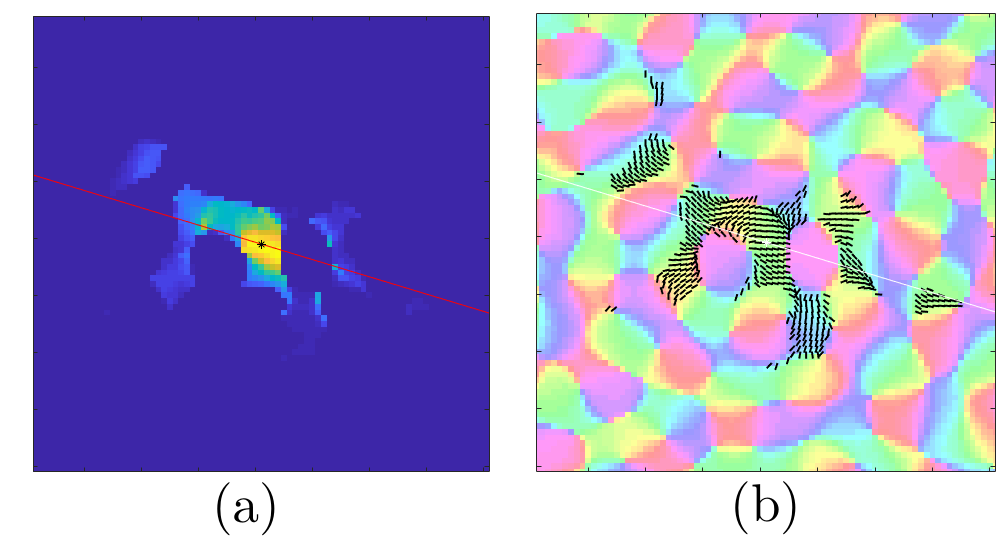}
  \caption{Propagation of the neural activity through the discretized heat equation associated to the graph Laplacian on $\mathcal{G}_\Theta$ with initial datum $\delta_{(0,0)}$. (a) The updated 2D kernel around $(0,0)$, with a black asterisk showing the starting point $(0,0)$, and a red line highlighting the orientation $\Theta(0,0)$. (b) At each point $(x,y)$ where the kernel exceeds a threshold, the corresponding orientation $\Theta(x,y)$, as in Figure \ref{Kpw}.}\label{Lpw}
 \end{figure}
 
 More simulations can be found in our work \cite{neuro}, including an example on curvature-selective neurons as well as the application of our construction to a family of \emph{numerically known} filters obtained through an optimization algorithm.

\section*{Acknowledgments}
 The authors have been supported by Horizon 2020 Project ref. 777822: GHAIA and PRIN 2015 ``Variational and perturbative aspects of nonlinear differential problems''.


\begin{thebibliography}{9}
\addcontentsline{toc}{chapter}{Bibliography}

\bibitem{abb}  A. A. Agrachev, D. Barilari, U. Boscain, \textit{A Comprehensive Introduction to sub-Riemannian Geometry}, Cambridge University Press, Cambridge (2019).

\bibitem{bosking} W. Bosking, Y. Zhang , B. Schoenfield, D. Fitzpatrick, \textit{Orientation selectivity and the arrangement of horizontal connections in tree shrew striate cortex}, J Neurosci 17(6), 2112-2127 (1997).

\bibitem{bresscow03} P. C. Bressloff, J. D. Cowan, \textit{The functional geometry of local and long-range connections in a model of V1}, J. Physiol. Paris, 97, 2-3, 221-236 (2003).

\bibitem{bresscow02}  P.C. Bressloff, J.D. Cowan, M. Golubitsky, P.J. Thomas, M.C. Wiener, \textit{What Geometric Visual Hallucinations Tell Us about the Visual Cortex}, Neural Computation, 14, 473-491 (2002).

\bibitem{graphappr}  D. Burago, S. Ivanov, Y. Kurylev, \textit{A graph discretization of the Laplace-Beltrami operator}, J. Spectr. Theory 4, 675-714 (2014).

\bibitem{cittimanfr} G. Citti, M. Manfredini, \textit{Implicit function theorem in Carnot-Carath\'{e}odory spaces}, Comm. in Cont. Math., Vol. 8, 5, 657-680 (2006).

\bibitem{cs06} G. Citti, A. Sarti, \textit{A Cortical Based Model of Perceptual Completion in the Roto-Translation Space}, Journal of Mathematical Imaging and Vision archive, Vol. 24, no. 3, p.307-326 (2006).

\bibitem{neuromat} G. Citti, A. Sarti (eds.), \textit{Neuromathematics of Vision}, Lecture Notes in Morphogenesis, Springer (2014).

\bibitem{coiflaf} R. R. Coifman, S. Lafon, \textit{Diffusion maps}, Appl. Comput. Harmon. Anal. 21, 5-30 (2006).

\bibitem{daug} J. G. Daugman, \textit{Uncertainty relation for resolution in space, spatial frequency, and orientation optimized by two-dimensional visual cortical filters}, J. Opt. Soc. Am. A2, 1160-1169 (1985).

\bibitem{deng} C.-X. Deng, S. Li, Z.-X. Fu, \textit{The reproducing kernel Hilbert space based on wavelet transform}, Proceedings of the 2010 International Conference on Wavelet Analysis and Pattern Recognition, Qingdao, 370-374 (2010).

\bibitem{federer} H. Federer, \textit{Geometric Measure Theory}, Springer-Verlag (1969).

\bibitem{field} D. J. Field, A. Hayes, R. F. Hess, \textit{Contour integration by the human visual system: evidence for a local “association field”}, Vision Res 33 173-193, (1993).

\bibitem{rectif} B. Franchi, R. Serapioni, F. Serra Cassano, \textit{Rectifiability and perimeter in the Heisenberg group}, Math. Ann. 321, 479-531 (2001).

\bibitem{gilbert} C. D. Gilbert, A. Das, M. Ito, M. Kapadia, G. Westheimer, \textit{Spatial integration and cortical dynamics}, Proceedings of the National Academy of Sciences USA, Vol. 93, 615-622 (1996).

\bibitem{gilwie} C. D. Gilbert, T.N. Wiesel, \textit{Morphology and intracortical projections of functionally identified neurons in cat visual cortex}, Nature 280, 120-125 (1979).

\bibitem{galri} M. Galli, M. Ritor\'{e}, \textit{Existence of isoperimetric regions in contact sub-Riemannian manifolds}, Journal of Mathematical Analysis and Applications, Vol. 397, Issue 2, 697-714 (2013).

\bibitem{gilwie89} C. D. Gilbert, T.N. Wiesel, \textit{Columnar specificity of intrinsic horizontal and corticocortical connections in cat visual cortex}, J. Neurosci. 9, 2432-2442 (1989).

\bibitem{gilwu} C. D. Gilbert, L. Wu, \textit{Top-down influences on visual processing}, Nature Reviews Neuroscience 14, 350-363 (2013).

\bibitem{hausd} F. Hausdorff, \textit{Dimension und \"{a}usseres Mass}, Mathematische Annalen, 79 (1-2), 157-179, 1918.

\bibitem{hoffman} W.C. Hoffman, \textit{The visual cortex is a contact bundle}, Appl. Math. Comput. 32, 137-167 (1989).

\bibitem{HW} D. H. Hubel, T. N. Wiesel, \textit{Receptive fields, binocular interaction and functional architecture in the cat visual cortex}, J. Physiol. (London) 160, 106-154 (1962).

\bibitem{hubel} D. H. Hubel, \textit{Eye, brain, and vision}, New York, WH Freeman (Scientific American Library) (1987).

\bibitem{jonpal} J. P. Jones, L. A. Palmer, \textit{An evaluation of the two-dimensional Gabor filter model of simple receptive fields in cat striate cortex}, J. Neurophysiol. 58, 1233-1258 (1987).

\bibitem{reciprocal} Z. F. Kisvarday, U. T. Eysel, \textit{Cellular organization ofreciprocal patchy networks in layer III of cat visual cortex(area 17)}, Neuroscience 46, 275-286 (1992).

\bibitem{lee} T. S. Lee, \textit{Image Representation Using 2D Gabor Wavelets}, IEEE Transactions on Pattern Analysis and Machine Intelligence, Vol 18, No. 10 (1996).

\bibitem{graph} B. L\'{e}vy, \textit{Laplace-Beltrami eigenfunctions: towards an algorithm that understands geometry}, Proc. of Shape Modeling and Applications,
page 13 (2006).

\bibitem{complex} L. M. Martinez, J.-M. Alonso, \textit{Complex receptive fields in primary visual cortex}, Neuroscientist, 9(5), 317-331 (2003).

\bibitem{neuro} N. Montobbio, G. Citti, A. Sarti, \textit{Receptive profiles induce functional architecture of V1}, J Comput Neurosci (2019).

\bibitem{mumford} D. Mumford, \textit{Elastica and computer vision}, in \textit{Algebraic Geometry and its Applications}, 507-518. ed. C. Bajaj, Springer-Verlag (1993).

\bibitem{nsw} A. Nagel, E. M. Stein, S. Wainger, \textit{Balls and metrics defined by vector fields I: Basic properties}, Acta Math. 155, 103-147 (1985).

\bibitem{petitot} J. Petitot, \textit{Neurog\'{e}om\'{e}trie de la vision - Mod\`{e}les math\'{e}matiques et physiques des architectures fonctionnelles}, \'{E}ditions de l'\'{E}cole Polytechnique (2008).

\bibitem{petitond} J. Petitot, Y. Tondut, \textit{Vers une neuro-g\'{e}om\'{e}trie. Fibrations corticales, structures de contact et contours subjectifs modaux}, Math\'{e}matiques , Informatique et Sciences Humaines, vol. 145, 5-101, EHESS, Paris (1999).

\bibitem{rkwavelet}  D. Pravica, N. Randriampiry, M. Spurr, \textit{Reproducing kernel bounds for an advanced wavelet frame via the theta function}, Appl. Comput. Harmon. Anal., Vol.33, 79-108 (2012).

\bibitem{graphsegm} M. Reuter, S. Biasotti, D. Giorgi, G. Patan\`{e}, M. Spagnuolo, \textit{Discrete Laplace-Beltrami operators for shape analysis and segmentation}, Computers \& Graphics, 33 (3), 381-390 (2009).

\bibitem{edge-stat} G. Sanguinetti, G. Citti, A. Sarti, \textit{A model of natural image edge co-occurrence in the rototranslation group}, J. Vis. 10(14) (2010).

\bibitem{scp} A. Sarti, G. Citti, J. Petitot, \textit{The symplectic structure of the visual cortex}, Biological Cybernetics, Volume 98, Issue 1, 33-48 (2008).

\bibitem{sturm95b} K.-T. Sturm, \textit{On the geometry defined by Dirichlet forms}, Seminar on Stochastic Analysis, Random Fields and Applications (E. Bolthausen et al., eds.) 231-242. Birkh\"{a}user, Boston (1995).

\bibitem{sturm98} K.-T. Sturm, \textit{Diffusion processes and heat kernels on metric spaces}, Ann. Probab. 26(1), 1-55 (1998).

\bibitem{vision} B. Wandell, \textit{Foundations of Vision: Behavior, Neuroscience and Computation}, Sinauer Associates Inc. (1995).

\bibitem{warner} F. W. Warner, \textit{Foundations of Differentiable Manifolds and Lie Groups}, Springer-Verlag, Berlin (1983).

\bibitem{quasimetric} W. A. Wilson, \textit{On Quasi-Metric Spaces}, American Journal of Mathematics, 53(3), 675 (1931). 

\bibitem{yeh} J. Yeh, \textit{Real Analysis. Theory of Measure and Integration}, World Scientific Publishing Company (2006).

\bibitem{yenfinkel} S. C. Yen, L. H. Finkel, \textit{Extraction of perceptually salient contours by striate cortical networks}, Vision Res 38(5):719-741 (1998).

\bibitem{zucker} S.W. Zucker, \textit{Differential geometry from the Frenet point of view: boundary detection, stereo, texture and color}, in: N. Paragios, Y. Chen, O. D. Faugeras (eds.), \textit{Handbook of Mathematical Models in Computer Vision}, 357-373. Springer, US (2006).

\end{thebibliography}
\end{document}